\documentclass[a4paper,11pt,reqno]{amsart}
\usepackage{amssymb,amsmath,hyperref,comment}
\title[On string functions of the generalized parafermionic theories]{On string functions of the generalized\\ parafermionic theories, mock theta functions, and\\false theta functions}

\author{Nikolay E. Borozenets}
\address{Leonhard Euler International Mathematical Institute, Saint Petersburg State University, Saint Petersburg, Russia, 199178}
\email{nikolayborozenets.spbumcs@gmail.com}

\author{Eric T. Mortenson}
\address{Department of Mathematics and Computer Science, Saint Petersburg State University, Saint Petersburg,  Russia, 199178}
\email{etmortenson@gmail.com}

\newcommand\sg{\operatorname{sg}}
\newtheorem{theorem}{Theorem}
\newtheorem{lemma}[theorem]{Lemma}
\newtheorem{corollary}[theorem]{Corollary}
\newtheorem{proposition}[theorem]{Proposition}
\theoremstyle{definition}
\newtheorem{definition}[theorem]{Definition}
\newtheorem{remark}[theorem]{Remark}

\numberwithin{theorem}{section} 
\numberwithin{equation}{section}

\newcommand{\C}{\mathbb{C}}

\newcommand{\R}{\mathbb{R}}
\newcommand{\Z}{\mathbb{Z}}

\newcommand{\im}{\textnormal{Im}}

\newcommand{\sgn}{\operatorname{sgn}}

\newenvironment{psmallmatrix}
  {\left(\begin{smallmatrix}}
  {\end{smallmatrix}\right)}

\makeatletter
\@namedef{subjclassname@2020}{\textup{2020} Mathematics Subject Classification}
\makeatother

\dedicatory{For George Andrews and Bruce Berndt in honor of their 85th birthdays}
\begin{document}

\date{20 January 2025}



\begin{abstract}
Kac and Wakimoto introduced the admissible highest weight representations in order to classify all modular invariant representations of the Kac--Moody algebras. For the Kac--Moody algebra $A_1^{(1)}$ the string functions of admissible representations are allowed to have certain rational levels and were realized by Ahn, Chung, and Tye as the characters of the generalized Fateev--Zamolodchikov parafermionic theories.  For the \text{$1/2$-level} string functions, we present their mixed mock modular properties as well as elegant mock theta conjecture-like identities involving two mock theta functions from Ramanujan's Lost Notebook.  In addition, we demonstrate that the negative level string functions can be evaluated in terms of false theta functions. 
\end{abstract}

\maketitle

\tableofcontents
 
\section{The introduction}
 
Mock modular forms have recently been appearing in various contexts in representation theory and mathematical physics, for example in the relation to Mathieu group $M_{24}$ and Umbral Moonshine \cite{CDH, EOT}, and Kac--Wakimoto's supercharacters \cite{KW01, KW14, KW16adv, KW16izv}. In this paper we present a new appearence of mock modularity in the connection to modular invariant representations of the Kac--Moody algebra $A_1^{(1)}$ and generalized Fateev--Zamolodchikov parafermionic theories.

Following Kac and Wakimoto \cite{KW88} in Section \ref{subsection:admissiblerepresentations} we define the admissible highest weight representations of the Kac--Moody algebra $A_1^{(1)}$ and introduce the string functions of these representations \cite{KW90}. In Section \ref{subsection:calcsrtfunc} we present examples of computations of string functions of integrable representations. In Section \ref{subsection:parafermionth} we overview how Ahn, Chung, and Tye \cite{ACT} realized these string functions as the characters of the generalized parafermionic theories. In Section \ref{subsection:pstrfunchecketype} we define Hecke-type double-sums, and describe how it is possible to write the string functions in terms of them. In Section \ref{subsection:mockthetaappellfunc} we introduce Ramanujan's mock theta functions and Appell functions, which we will show appear as building blocks of the $1/2$-level string functions. In Section \ref{subsection:heckedoubsumappellfunc} we present the formulas by Mortenson and Zwegers \cite{MZ} converting Hecke-type double-sums of positive discriminant to Appell function form.  These double-sum formulas in principle apply to all positive fractional-level string functions, but in certain cases one is able to take advantage of the more symmetric formulas found in \cite[Theorems $1.3$, $1.4$]{HM}. In Section \ref{subsection:heckedoubsumfalsetheta} we introduce the notion of false theta functions, which we will show to be building blocks of the negative fractional-level string functions, and also present the formulas by Mortenson \cite{Mo24A} converting Hecke-type double-sums of negative discriminant to false theta function form.

\subsection{String functions of admissible highest weight representations} \label{subsection:admissiblerepresentations}

Kac and Wakimoto introduced admissible highest weight representations as a conjectural classification of all modular invariant representations \cite{KW88}. This conjecture is still open \cite{KW24}. We will focus on the case of the Kac--Moody algebra $A_1^{(1)}$ for which the Kac--Wakimoto's conjecture is known to be true. We let $p \geq 1$, $p^{\prime} \geq 2$ be coprime integers, and we define the admissible level to be
\begin{equation} \label{equation:admlevel}
N:=\frac{p^{\prime}}{p}-2.
\end{equation}
We then denote by $L(\lambda)$ an admissible $A_{1}^{(1)}$ highest weight representation of highest weight 
\begin{equation}\label{equation:highestweight}
\lambda =  \lambda^{I} - (N+2)\lambda^{F},
\end{equation}
where $\lambda^{I}$ and $\lambda^{F}$ are two integrable weights of levels $p'-2$ and $p-1$ respectively, that is, for $0 \leq \ell \leq p'-2$ and $0 \leq k \leq p-1$ we have
\begin{align*}
\lambda^{I} &= (p'-\ell-2) \Lambda_0 + \ell \Lambda_1,\\
\lambda^{F} &= (p-k-1) \Lambda_0 + k \Lambda_1.
\end{align*}
where $\Lambda_0$ and $\Lambda_1$ are the fundamental weights of $A_{1}^{(1)}$. In this paper we will consider only the case of $k=0$, so that the spin, the coefficient of $\Lambda_1$ in \eqref{equation:highestweight}, is equal to $\ell$ and hence is a positive integer. Note that when $p= 1$, admissible representations reduce to integrable ones \cite{KP}, that is, the fractional part vanishes $\lambda^{F} = 0$.

Let $q := e^{2\pi i \tau}$ with $\im(\tau) >0$ and $z \in \mathbb{C}\backslash \{0\}$. The character for irreducible highest weight representation of admissible highest weight \eqref{equation:highestweight} is
\begin{equation} \label{eq:chdef}
\chi_{\ell}^N(z;q):=\textup{Tr}_{L(\lambda)} \left(q^{s_{\lambda}-d}z^{-\frac{1}{2}J_0 }\right),
\end{equation}
where $d$ is a derivation, $J_0$ is $\operatorname{U}(1)$-charge operator and 
\begin{equation*}
s_{\lambda}:=-\frac{1}{8}+\frac{(\ell+1)^2}{4(N+2)}.
\end{equation*}
By the Weyl--Kac formula it is possible to express the character as
\begin{equation}\label{equation:WK-formula}
\chi_{\ell}^N(z;q)=\frac{\sum_{\sigma=\pm 1}\sigma \Theta_{\sigma (\ell+1),p^{\prime}}(z;q^{p})}
{\sum_{\sigma=\pm 1}\sigma\Theta_{\sigma,2}(z;q)},
\end{equation}
where we denote the theta function as
\begin{equation}\label{equation:SW-thetaDef}
\Theta_{n,m}(z;q):=\sum_{j\in\mathbb{Z}+n/2m}q^{mj^2}z^{-mj}.
\end{equation}
Using \eqref{equation:WK-formula} Kac and Wakimoto showed that the characters form a vector-valued Jacobi form.

Let us denote the energy eigenspaces with respect to $-d$ as
\begin{equation*}
L(\lambda)_{(n)} := \{v \in L(\lambda) \ | \ (-d) v = n v \  \}
\end{equation*}
 for $n \in \Z_{\geq 0}$ and the weight space associated to the weight 
 \begin{equation} \label{eq:arbweight}
 \mu = (N-m)\Lambda_0+m\Lambda_1 
 \end{equation}
 as
\begin{equation*}
L(\lambda)_{[m]} := \{v \in L(\lambda) \ | \ J_0 v = m v \  \}.
\end{equation*}
For admissible highest weight \eqref{equation:highestweight} and arbitrary weight \eqref{eq:arbweight} we define the string function as
\begin{equation*}
c^{\lambda}_{\mu} = c^{N-\ell,\ell}_{N-m,m} = C^{N}_{m,\ell}(q) := q^{s_{\lambda,\mu}} \sum_{n\geq 0} \dim (L(\lambda)_{[m]} \cap L(\lambda)_{(n)}) q^n,
\end{equation*}
where
\begin{equation*}
s_{\lambda,\mu} := s_{\lambda} - \frac{m^2}{4N}.   
\end{equation*}
We also define the additional notation
\begin{equation}
\mathcal{C}_{m,\ell}^{N}(q) := q^{-s_{\lambda,\mu}}C_{m,\ell}^{N}(q) \in \Z[[q]].
\end{equation} 
From the definition we have the Fourier expansion
\begin{equation} \label{equation:fourcoefexp}
\chi_{\ell}^N (z,q)=\sum_{m\in 2\mathbb{Z}+\ell}
C_{m,\ell}^{N}(q) q^{\frac{m^2}{4N}}z^{-\frac{1}{2}m}.
\end{equation} 
We have the following symmetry of string functions \cite[(3.4), (3.5)]{SW}, \cite[(2.40)]{ACT}
\begin{align*}
C_{m,\ell}^{N}(q) &= C_{-m,\ell}^{N}(q),\\
C_{m,\ell}^{N}(q) &= C_{N-m,N-\ell}^{N}(q).
\end{align*}
For the integral level $N$ we have the periodicity property \cite[(3.5)]{SW}
\begin{equation} \label{eq:inglevelperiod}
C_{m,\ell}^{N}(q) = C_{m+2N,\ell}^{N}(q).
\end{equation}
and hence from \eqref{equation:fourcoefexp} the theta-expansion
\begin{equation}\label{eq:intlevelthetadecomp}
\chi_{\ell}^N(z,q)=\sum_{\substack{0\le m <2N\\m \in 2\Z + \ell}}C_{m,l}^{N}(q)\Theta_{m,N}(z,q).
\end{equation}

\subsection{Computation of string functions}
\label{subsection:calcsrtfunc}
Important problems in the representation theory of Kac--Moody algebras are the modular transformation properties and the explicit calculation of the string functions. Kac and Peterson \cite{KP} gave several examples of elegant evaluations in terms of theta functions of string functions of integrable highest weight representations of $A_1^{(1)}$. Recall the Dedekind eta-function,
\begin{equation*} 
\eta(q)=\eta(\tau) :=q^{1/24}\prod_{n\ge 1}(1-q^{n}),
\end{equation*}
where we use variables $q$ and $\tau$ depending on the contest. For example, Kac and Peterson \cite{KP} showed
\begin{gather*}
c^{01}_{01} = \eta(\tau)^{-1},\\
c^{11}_{11} = \eta(\tau)^{-2}\eta(2\tau),\\
c^{40}_{22} = \eta(\tau)^{-2} \eta(6\tau) \eta(12\tau)^{2},\\
c^{40}_{40} - c^{40}_{04} = \eta(2\tau)^{-2}.
\end{gather*}
Kac and Peterson appeal to modularity to prove the string function identities \cite[p. 220]{KP}. Specifically, they use the transformation law for string functions under the full modular group, together with the calculation of the first few terms in the Fourier expansions of the string functions. As described in Remark \ref{remark:morthicksymfabc} Mortenson \cite{Mo24B} improved the calculations of Kac and Peterson without  appealing to modularity and obtained the complete list of explicit formulas in terms of theta functions for all string functions of levels $N\in\{1,2,3,4\}$. Additional calculations can be found in \cite{MPS}.  Also, string function for the $A_1^{(1)}$ were presented as $q$-hypergeometric expressions by Lepowsky and Primc \cite{LP}. Schilling and Warnaar found fermionic or constant-sign expressions for string functions of admissible highest weight representations of $A_1^{(1)}$ \cite{SW}.

\subsection{String functions as the characters of generalized parafermionic theories}
\label{subsection:parafermionth}

Fateev and Zamolodchikov constructed the $Z_N$ parafermionic (PF) theory  \cite{ZF85}, which can be identified with coset CFT constructed from $\operatorname{SL}(2)_N$ Wess--Zumino--Witten (WZW) theories
\begin{equation} \label{equation:ZNcoset}
Z_N = \operatorname{SL}(2)_N / \operatorname{U}(1).
\end{equation}
Using the coset realization \eqref{equation:ZNcoset} and the construction of admissible highest weight as defined in Section \ref{subsection:admissiblerepresentations}, Ahn, Chung, and Tye \cite{ACT} generalized the $Z_N$ PF algebra from the integer level to the admissible level \eqref{equation:admlevel}. Note that in the case of a fractional-level, such generalized PF theories are non-unitary.

The $\operatorname{SL}(2)_N$ Hilbert spaces $\mathcal{H}_{N,\ell}$ of spin $\ell$ can be decomposed into Hilbert spaces of states of fixed $\operatorname{U(1)}$-charge $m \in 2\Z + \ell$ as
\begin{equation} \label{equation:SLintoVir}
\mathcal{H}_{N,\ell} = \bigoplus_{m \in 2\Z + \ell} \mathcal{H}_{N,\ell,m}.
\end{equation}
Ahn, Chung, and Tye \cite{ACT} factored $\mathcal{H}_{N,\ell,m}$ into PF Hilbert space and that of the boson,
\begin{equation} \label{equation:VirIntoPFandB}
\mathcal{H}_{N,\ell,m} = \mathcal{H}_{N,\ell,m}^{\operatorname{PF}} \otimes  \mathcal{H}_{N,m}^{\operatorname{b}}.
\end{equation}
Let us denote the PF character by $e_{m,\ell}^{N}(q)$,  from \eqref{equation:SLintoVir} and \eqref{equation:VirIntoPFandB} we get
\begin{equation} \label{equation:string-def}
\chi_{\ell}^N (z,q)=\sum_{m\in 2\mathbb{Z}+\ell}
e_{m,\ell}^{N}(q) \cdot \frac{q^{\frac{m^2}{4N}}z^{-\frac{1}{2}m}}{\eta(q)}.
\end{equation} 
From \eqref{equation:fourcoefexp} we see that
\begin{equation*}
e_{m,\ell}^{N}(q) = \eta(q) C_{m,\ell}^{N}(q).   
\end{equation*}
Ahn, Chung, and Tye considered the PF characters as the building blocks of the characters of the diagonal $A_1^{(1)}$ coset theories of rational levels as well as certain superconformal field theories, for details see \cite{ACT}.

\subsection{String functions in Hecke-type double-sum form}
\label{subsection:pstrfunchecketype}

We recall the $q$-Pochhammer notation
\begin{equation*}
(x)_n=(x;q)_n:=\prod_{i=0}^{n-1}(1-q^ix), \ \ (x)_{\infty}=(x;q)_{\infty}:=\prod_{i\ge 0}(1-q^ix),
\end{equation*}
and the theta function
\begin{equation}\label{equation:JTPid}
j(x;q):=(x)_{\infty}(q/x)_{\infty}(q)_{\infty}=\sum_{n=-\infty}^{\infty}(-1)^nq^{\binom{n}{2}}x^n,
\end{equation}
where the last equality is the Jacobi triple product identity.  One notes that the two definitions for theta functions are equivalent via the identity
\begin{equation*}
\Theta_{n,m}(z;q)=z^{-\frac{n}{2}}q^{\frac{n^2}{4m}}j\left ( -q^{n+m}z^{-m};q^{2m}\right  ). 
\end{equation*}

Using the classical partial fraction expansion for the reciprocal of Jacobi's theta product,
\begin{equation*}
\frac{1}{j(z;q)}=\frac{1}{(q;q)_{\infty}^3}\sum_{n\in\mathbb{Z}}\frac{(-1)^nq^{\binom{n+1}{2}}}{1-q^{n}z}
\end{equation*}
and the Weyl--Kac formula \eqref{equation:WK-formula}, one is able to extract \cite[(3.8)]{SW}:
\begin{align}
\begin{split}\label{equation:fractional-string}
\mathcal{C}_{m,\ell}^{N}(q)&=\frac{1}{(q)_{\infty}^3}
 \left \{ \sum_{\substack{i\ge 0 \\ j\ge 0}}-\sum_{\substack{i< 0 \\ j< 0}}\right \}
 (-1)^{i}q^{\frac{1}{2}i(i+m)+p^{\prime}j(pj+i)+\frac{1}{2}(\ell+1)(2pj+i)}\\
&\qquad -\frac{1}{(q)_{\infty}^3}
 \left \{ \sum_{\substack{i\ge 0 \\ j > 0}}-\sum_{\substack{i< 0 \\ j\le 0}}\right \}
 (-1)^{i}q^{\frac{1}{2}i(i+m)+p^{\prime}j(pj+i)-\frac{1}{2}(\ell+1)(2pj+i)}.
 \end{split}
 \end{align}
For similar derivations see also \cite[Section 2.4]{ACT} and \cite[Proposition 3]{L}. We recall the definition for Hecke-type double-sums.
\begin{definition} \label{definition:fabc-def}  Let $x,y\in\mathbb{C} \backslash \{0\}$. Then
\begin{equation} \label{equation:fabc-def2}
f_{a,b,c}(x,y;q):=\left( \sum_{r,s\ge 0 }-\sum_{r,s<0}\right)(-1)^{r+s}x^ry^sq^{a\binom{r}{2}+brs+c\binom{s}{2}},
\end{equation}
where we define the discriminant to be
\begin{equation*}
D:=b^2-ac.
\end{equation*}
\end{definition}

We can rewrite \eqref{equation:fractional-string} in terms of Hecke-type double-sums.

\begin{proposition} \label{proposition:modStringFnHeckeForm} Let $p'\geq 2$, $p\geq 1$ be coprime integers, $0\leq \ell \leq p'-2$ and $m\in 2\Z +\ell$. We have
\begin{equation*}
\mathcal{C}_{m,\ell}^{N}(q)
 =\frac{1}{(q)_{\infty}^3}\left ( f_{1,p^{\prime},2pp^{\prime}}(q^{1+\frac{m+\ell}{2}},-q^{p(p^{\prime}+\ell+1)};q)
 -f_{1,p^{\prime},2pp^{\prime}}(q^{\frac{m-\ell}{2}},-q^{p(p^{\prime}-(\ell+1))};q)\right).
\end{equation*}      
\end{proposition}

\begin{remark} \label{remark:shortintstrfunc} In the case of positive integer level $N>0$, we have the compact form \cite[Example $1.3$]{HM}
\begin{equation*}
\mathcal{C}_{m,\ell}^{N}(q)=\frac{1}{(q)_{\infty}^3}
f_{1,1+N,1}(q^{1+\tfrac{1}{2}(m+\ell)},q^{1-\tfrac{1}{2}(m-\ell)};q).
\end{equation*}
\end{remark}

\subsection{Mock theta functions and Appell functions} 
\label{subsection:mockthetaappellfunc}

In his deathbed letter to Hardy in 1920, Ramanujan introduced his so-called mock theta functions. Ramanujan offered 17 examples of mock theta functions in this letter, which he grouped by order, a notion that he did not define. Many examples of mock theta functions also appeared in the Ramanujan's ``Lost Notebook'' \cite{Ram} discovered by George Andrews in the library at Trinity College, Cambridge in 1976. As an example we present Ramanujan's classical second-order mock theta functions
\begin{equation} \label{equation:secondordmocktheta}
\mu(q):=\sum_{n\ge 0}\frac{(-1)^nq^{n^2}(q;q^2)_n}{(-q^2;q^2)_{n}^2}, \ 
A(q):=\sum_{n\ge 0}\frac{q^{(n+1)^2}(-q;q^2)_n}{(q;q^2)_{n+1}^2}, 
\end{equation}
both of which appearing in the ``Lost Notebook''  on page 8 of \cite{Ram}.   

Ramanujan presented all his examples in so-called Eulerian form, that is, as $q$-series similar in ``shape'' to basic hypergeometric, also known as $q$-hypergeometric, series. To facilitate studying mock theta functions, it is useful to translate the Eulerian form into other representations: Appell function form, Hecke-type double-sums, and Fourier coefficients of meromorphic Jacobi forms, all of which were unified by Zwegers in his celebrated thesis \cite{Zw02}. The general notion of (mixed) mock modular forms can be found in \cite{DMZ}.

Early attempts to unify Eulerian forms and Appell functions led to identities between the mock theta functions \cite{McInt} and also helped to determine modular properties \cite{Wat}.  Identities found in the ``Lost Notebook,'' such as the mock theta conjectures \cite{H1}, led to expressing mock theta functions as Hecke-type double-sums \cite{A1986}.  

The mock theta (ex-)conjectures \cite{AG, H1} were a collection of ten identities found in the ``Lost Notebook'' that expressed fifth order mock theta functions in terms of a building block Eulerian form and a single quotient of theta functions.  The (ex-)conjecture for the fifth order function $f_0(q)$ reads
\begin{equation} \label{equation:mockthetaconj}
f_0(q):=\sum_{n\ge 0}\frac{q^{n^2}}{(-q)_n}=2-2\sum_{n\ge 0}\frac{q^{10n^2}}{(q^2;q^{10})_{n+1}(q^8;q^{10})_{n}}+\frac{J_{5}J_{5,10}}{J_{1,5}}.
\end{equation}
Many identities in the Lost Notebook express mock theta functions in terms of single quotients of theta functions, with there being no apparent explanation for the phenomenon \cite{Mo24C, MU24}.

Following Hickerson and Mortenson \cite{HM} we introduce the Appell function as
\begin{equation*}
m(x,z;q):=\frac{1}{j(z;q)}\sum_{r\in\Z}\frac{(-1)^rq^{\binom{r}{2}}z^r}{1-q^{r-1}xz}.
\end{equation*}

We will use the following short-hand notation,
\begin{equation} \label{equation:shortnottheta}
J_{a,b}:=j(q^a;q^b), \ \ \overline{J}_{a,b}:=j(-q^a;q^b), \ {\text{and }}J_a:=J_{a,3a}=\prod_{i\ge 1}(1-q^{ai}),
\end{equation}
where $a,b$ are positive integers.

Ramanujan's classical second-order mock theta functions \eqref{equation:secondordmocktheta} can be expressed in terms of Appell functions as
\begin{equation}\label{equation:muAppellForm}
\mu(q)=4m(-q,-1;q^4)-\frac{J_{2,4}^4}{J_1^3}, \ A(q)=-m(q,q^2;q^4).
\end{equation}
See \cite[Section $5$]{HM} for more examples.

\subsection{Hecke-type double-sums and Appell functions}
\label{subsection:heckedoubsumappellfunc}

In \cite{HM, MZ}, Hecke-type double-sums \eqref{equation:fabc-def2} with positive discriminant $D$ are extensively studied.  Expansions are obtained that express the double-sums in terms of theta and Appell functions. We first define the following expression involving Appell functions.

\begin{definition} Let $a,b,$ and $c$ be positive integers with $D:=b^2-ac>0$.  Then
\begin{align} 
m_{a,b,c}(x,y,z_1,z_0;q)
&:=\sum_{t=0}^{a-1}(-y)^tq^{c\binom{t}{2}}j(q^{bt}x;q^a)m\Big (-q^{a\binom{b+1}{2}-c\binom{a+1}{2}-tD}\frac{(-y)^a}{(-x)^b},z_0;q^{aD}\Big ) \label{equation:mabc-def}\\
&\ \ \ \ \ +\sum_{t=0}^{c-1}(-x)^tq^{a\binom{t}{2}}j(q^{bt}y;q^c)m\Big (-q^{c\binom{b+1}{2}-a\binom{c+1}{2}-tD}\frac{(-x)^c}{(-y)^b},z_1;q^{cD}\Big ).\notag
\end{align}
\end{definition}

Mortenson and Zwegers \cite{MZ} obtained a decomposition for the general form (\ref{equation:fabc-def2}) with positive discriminant $D$.  

\begin{theorem}\cite[Corollary 4.2]{MZ}\label{theorem:posDisc} Let $a,b,$ and $c$ be positive integers with $D:=b^2-ac>0$. For generic $x$ and $y$, we have
\begin{align*}
& f_{a,b,c}(x,y;q)=m_{a,b,c}(x,y,-1,-1;q)+\frac{1}{j(-1;q^{aD})j(-1;q^{cD})}\cdot \vartheta_{a,b,c}(x,y;q),
\end{align*}
where
\begin{align*}
&\vartheta_{a,b,c}(x,y;q):=
\sum_{d^*=0}^{b-1}\sum_{e^*=0}^{b-1}q^{a\binom{d-c/2}{2}+b( d-c/2 ) (e+a/2  )+c\binom{e+a/2}{2}}(-x)^{d-c/2}(-y)^{e+a/2}\\
&\cdot\sum_{f=0}^{b-1}q^{ab^2\binom{f}{2}+\big (a(bd+b^2+ce)-ac(b+1)/2 \big )f} (-y)^{af}
\cdot j(-q^{c\big ( ad+be+a(b-1)/2+abf \big )}(-x)^{c};q^{cb^2})\\
&\cdot j(-q^{a\big ( (d+b(b+1)/2+bf)(b^2-ac) +c(a-b)/2\big )}(-x)^{-ac}(-y)^{ab};q^{ab^2D})\\
&\cdot \frac{(q^{bD};q^{bD})_{\infty}^3j(q^{ D(d+e)+ac-b(a+c)/2}(-x)^{b-c}(-y)^{b-a};q^{bD})}
{j(q^{De+a(c-b)/2}(-x)^b(-y)^{-a};q^{bD})j(q^{Dd+c(a-b)/2}(-y)^b(-x)^{-c};q^{bD})}.
\end{align*}
Here $d:=d^*+\{c/2 \}$ and $e:=e^*+\{ a/2\}$, with  $0\le \{\alpha \}<1$ denoting fractional part of $\alpha$.
\end{theorem}

\begin{remark} \label{remark:morthicksymfabc} Hickerson and Mortenson \cite[Theorem $1.3$]{HM} obtained a decomposition for a symmetric form of (\ref{equation:fabc-def2}) with positive discriminant $D$: 
 \begin{align} \label{equation:decompsymfabc}
f_{n,n+p,n}(x,y,q)=m_{n,n+p,n}(x,y,q,-1,-1)+\frac{1}{\overline{J}_{0,np(2n+p)}}\cdot \vartheta_{n,p}(x,y,q),
\end{align}
where $\vartheta_{n,p}(x,y,q)$ is a sum consisting of quotients of theta functions. Using the decomposition \eqref{equation:decompsymfabc} and Remark \ref{remark:shortintstrfunc}, Mortenson \cite{Mo24B} obtained explicit formulas for string functions of integral level $N\in\{1,2,3,4\}$. For example, for $\ell \in \{0,1\}$ and $m \in 2\Z +\ell$ we have
\begin{equation*}
C_{m,\ell}^1(q)=
\frac{1}{(q)_{\infty}^3}\cdot f_{1,2,1}(q^{1+\tfrac{1}{2}(m+\ell)},q^{1-\tfrac{1}{2}(m-\ell)},q)
=\frac{q^{\frac14(m^2-\ell^2)}}{(q)_{\infty}}.
\end{equation*}
\end{remark}
\noindent As we will soon see, there is also another useful symmetric form where $a,c \mid b$ and $D=b^2-ac>0$, see \cite[Theorem $1.4$]{HM}.

\subsection{Hecke-type double-sums and false theta functions} 
\label{subsection:heckedoubsumfalsetheta}
False theta functions are theta functions but with the ``wrong signs'' firstly considered by Rogers \cite{Rog}. Let $r\in\mathbb{Z}$ and define 
\begin{equation*}
\sg(r):=\begin{cases}
1,&\textup{if} \ r\ge0\\
-1,&\textup{if} \ r<0.
\end{cases}
\end{equation*}
Then, if we write the theta function of definition \eqref{equation:JTPid} with incorrect signs we get
\begin{equation} \label{equation:ramanujanpartialtheta}
\sum_{n\in\mathbb{Z}}\sg(n)(-1)^{n}q^{\binom{n+1}{2}}
=\sum_{n=0}^{\infty}(-1)^{n}q^{\binom{n+1}{2}}-\sum_{n=-\infty}^{-1}(-1)^{n}q^{\binom{n+1}{2}}
=2\sum_{n=0}^{\infty}(-1)^{n}q^{\binom{n+1}{2}},
\end{equation}
that is, it turns out to be a partial theta function, i.e. functions resembling \eqref{equation:JTPid} but with the summation over $\Z$ replaced by a partial lattice (e.g. $n \geq n_0$ for $n_0 \in \Z$). Partial theta functions \eqref{equation:ramanujanpartialtheta} were also studied by Ramanujan in his ``Lost Notebook'' \cite{A81}. Nevertheless, in general, partial theta functions are not the same as false theta functions; however, most false theta functions that one encounters are a sum of two specializations of partial theta functions. Also recently, Bringmann and Nazaroglu introduced a framework for modular properties of false theta functions \cite{BN}.

Lastly, Mortenson \cite{Mo24A} obtained a decomposition for the general form (\ref{equation:fabc-def2}) with negative discriminant $D$.

\begin{theorem}\cite[Theorem $1.4$]{Mo24A} \label{theorem:negDisc}
Let $a,b,$ and $c$ be positive integers with $D:=b^2-ac<0$. For generic $x$ and $y$, we have
\begin{align} 
f_{a,b,c}&(x,y;q) \label{equation:mainIdentity}\\
&=\frac{1}{2}\Big (  \sum_{t=0}^{a-1}(-y)^tq^{c\binom{t}{2}}j(q^{bt}x;q^a)
\sum_{r\in\mathbb{Z}}\sg(r) \left (q^{a\binom{b+1}{2}-c\binom{a+1}{2}-tD}\frac{(-y)^a}{(-x)^b}\right )^r
q^{-aD\binom{r+1}{2}} \notag \\
&\ \ \ \ \ +\sum_{t=0}^{c-1}(-x)^tq^{a\binom{t}{2}}j(q^{bt}y;q^c)
\sum_{r\in\mathbb{Z}}\sg(r)\left (q^{c\binom{b+1}{2}-a\binom{c+1}{2}-tD}\frac{(-x)^c}{(-y)^b}\right )^rq^{-cD\binom{r+1}{2}}\Big ).\notag
\end{align}
\end{theorem}

\section{Main results}

In Section \ref{subsection:strfuncposadmlevel} we present the evaluation of the $1/2$-level string functions in terms of Ramanujan's second-order mock theta function \eqref{equation:secondordmocktheta}. We also demonstrate mixed mock modular properties on the whole modular group, and we prove mock theta conjecture-like identities for the $1/2$-level string functions. In Section \ref{subsection:strfuncnegadmlevel} we present the evaluation of string functions of negative admissible level in terms of false theta functions and also show the compact false theta function forms for the $(-1/2)$-level and $(-2/3)$-level string functions.

\subsection{The $1/2$-level string functions}
\label{subsection:strfuncposadmlevel}

For positive admissible level $N>0$ as defined in Section \ref{subsection:admissiblerepresentations} we have $p^{\prime}>2p$ and hence the discriminant of Hecke-type double-sums in Proposition \ref{proposition:modStringFnHeckeForm} is positive $D = (p')^2-2pp' > 0$. We can use Theorem \ref{theorem:posDisc} in order to obtain the representation of the string functions of positive admissible level $N$ in terms of Appell functions and theta functions. But 
\begin{equation*}
m_{1,p',2pp'}(x,y,-1,-1;q) \ \text{and} \ \vartheta_{1,p',2pp'}(x,y;q) 
\end{equation*}
can be undefined for the arguments
\begin{equation*}
(x,y) = (q^{1+\frac{m+\ell}{2}},-q^{p(p'+\ell+1)}) \ \text{and} \    (x,y) = (q^{\frac{m-\ell}{2}},-q^{p(p'-\ell-1)})    
\end{equation*}
as the denominator of some summands vanishes. We solve this problem for the $1/2$-level string functions.

\begin{theorem}\label{theorem:fractionalLevelPlus12} Let $(p,p')=(2,5),$ $0\le \ell\le 3$ and $m\in2\mathbb{Z}+\ell$.  We have that
\begin{multline*}
 (q)_{\infty}^3 \mathcal{C}_{m,\ell}^{1/2}(q) 
  = q^{\frac{1}{2}(m-\ell)}  j(q^{1+\ell};q^5)
 \left ( (-1)^{m} q^{\binom{m}{2}} \frac{1}{2}\mu(q) 
 +\sum_{k=0}^{m-1}(-1)^kq^{mk-\binom{k+1}{2}}\right )\\
 +(-1)^{m} q^{\frac{1}{2}(m-\ell)} q^{\binom{m}{2}} \Theta_{\ell}(q),
\end{multline*}
where
\begin{equation*}
\Theta_{\ell}:=
\begin{cases}
\frac{1}{2}\cdot \frac{J_{1}^3J_{10}^{3}}{J_{4}J_{5}}\cdot \frac{1}{J_{1,10}J_{8,20}}, & \textup{for} \ \ell \in \{0,3 \},\\
-\frac{1}{2}\cdot \frac{J_{1}^3J_{10}^{3}}{J_{4}J_{5}}\cdot \frac{1}{J_{3,10}J_{4,20}}, & \textup{for} \ \ell \in \{1,2 \},
\end{cases}
\end{equation*} 
where $\mu(q)$ is Ramanujan's classical second-order mock theta function \eqref{equation:secondordmocktheta}.
\end{theorem}

\begin{remark} \label{remark:summationnotation} When $b<a$ we follow the standard summation convention:
\begin{equation}
\sum_{r=a}^{b}c_r:=-\sum_{r=b+1}^{a-1}c_r, \ \textup{e.g.} \ 
\sum_{r=0}^{-1}c_r=-\sum_{r=0}^{-1}c_r=0. \label{equation:sumconvention}
\end{equation}
\end{remark}

For theta functions \eqref{equation:shortnottheta} we use the notation with the modular variable
\begin{equation} \label{eq:defmodthetafunc}
\mathcal{J}_{a,m}(\tau) = \mathcal{J}_{a,m}(q) :=q^{\frac{(m-2a)^2}{8m}} J_{a,m}, \ \ \overline{\mathcal{J}}_{a,m}(\tau) := q^{\frac{(m-2a)^2}{8m}} \overline{J}_{a,m}, \ \ \mathcal{J}_{m}(\tau):=\mathcal{J}_{m,3m},
\end{equation}
where we use variables $q$ and $\tau$ depending on the context.

For the $1/2$-level string functions with quantum number $m=0$ and even spin we can derive the mixed mock modular transformation properties on the whole modular group.

\begin{theorem} \label{theorem:mockmodtransfstrfunc25l0} We have
\begin{equation*}
\begin{pmatrix}
C^{1/2}_{0,0}\\
C^{1/2}_{0,2}
\end{pmatrix}(\tau+1) = \begin{pmatrix}
\zeta_{40}^{-1} & 0\\
0 & \zeta_{40}^{-9} \\
\end{pmatrix}
\begin{pmatrix}
C^{1/2}_{0,0}\\
C^{1/2}_{0,2}
\end{pmatrix}(\tau)    
\end{equation*}
and
\begin{multline*}
\begin{pmatrix}
C^{1/2}_{0,0}\\
C^{1/2}_{0,2}
\end{pmatrix}(\tau) = \sqrt{-i\tau} \cdot \frac{2}{\sqrt{5}}  \begin{pmatrix}
\sin\left(\frac{2\pi}{5}\right) & -\sin\left(\frac{\pi}{5}\right)  \\
-\sin\left(\frac{\pi}{5}\right) & -\sin\left(\frac{2\pi}{5}\right)
\end{pmatrix} \begin{pmatrix}
  C^{1/2}_{0,0}\\
  C^{1/2}_{0,2}
\end{pmatrix}\left(-\frac{1}{\tau}\right) \\
- \frac{i}{2} \cdot \frac{1}{\eta(\tau)^3}\cdot \begin{pmatrix}
\mathcal{J}_{1,5}\\
\mathcal{J}_{2,5}
\end{pmatrix}(\tau) \cdot \int_{0}^{i\infty} \frac{\eta(z)^3}{\sqrt{-i(z+\tau)}}dz.
\end{multline*}
\end{theorem}

We recall Ramanujan's classical second-order mock theta function  $A(q)$ \eqref{equation:secondordmocktheta}.  We have

\begin{corollary} \label{corollary:fractionalLevelPlus12-2ndA} Let $(p,p')=(2,5),$ $0\le \ell\le 3$ and $m\in2\mathbb{Z}+\ell$.  We have that
\begin{multline*}
 (q)_{\infty}^3 \mathcal{C}_{m,\ell}^{1/2}(q) 
 =-q^{\frac{1}{2}(m-\ell)}j(q^{1+\ell};q^5)
 \left ( (-1)^{m} q^{\binom{m}{2}} 2 A(-q) 
 -\sum_{k=0}^{m-1}(-1)^kq^{mk-\binom{k+1}{2}}\right )\\
 +(-1)^{m} q^{\frac{1}{2}(m-\ell)} q^{\binom{m}{2}} \Theta_{\ell}(q),
\end{multline*}
where
\begin{equation*}
\Theta_{\ell}:=
\begin{cases}
\frac{J_{1}^4J_{4}J_{8,20}}{J_{2}^4}, & \textup{for} \ \ell \in \{0,3 \},\\
- q\frac{J_{1}^4J_{4}J_{4,20}}{J_{2}^4}, & \textup{for} \ \ell \in \{1,2 \},
\end{cases}
\end{equation*} 
where $A(q)$ is Ramanujan's classical second-order mock theta function \eqref{equation:secondordmocktheta}.
\end{corollary}

We also have the immediate corollaries consisting of only one theta quotient and in this sense similar to the mock theta conjectures  \eqref{equation:mockthetaconj}. 
\begin{corollary} \label{theorem:stringfuncsimple0002} In terms of Ramanujan's second-order mock theta function  $\mu(q)$, we have
\begin{equation} \label{equation:strfunc2500short}
(q)_{\infty}^3 \mathcal{C}_{0,0}^{1/2}(q)
=\frac{1}{2}j(q;q^5)\mu(q)+\frac{1}{2}\cdot \frac{J_{1}^3J_{10}^3}{J_{4}J_{5}}\cdot \frac{1}{J_{1,10}J_{8,20}},
\end{equation}
and
\begin{equation}\label{equation:strfunc2502short}
(q)_{\infty}^3 \mathcal{C}_{0,2}^{1/2}(q)
=\frac{1}{2q}j(q^2;q^5)\mu(q)-\frac{1}{2q}\cdot \frac{J_{1}^3J_{10}^3}{J_{4}J_{5}}\cdot \frac{1}{J_{3,10}J_{4,20}}.
\end{equation}
\end{corollary}

\begin{corollary}\label{corollary:stringfuncsimple0002} 
In terms of Ramanujan's second-order mock theta function  $A(q)$, we have
\begin{equation} \label{equation:strfunc2500-Cor}
(q)_{\infty}^3 \mathcal{C}_{0,0}^{1/2}(q)
=-2j(q;q^5)A(-q)+ \frac{J_{1}^4J_{4}J_{8,20}}{J_{2}^4},
\end{equation}
and
\begin{equation}\label{equation:strfunc2502-Cor}
(q)_{\infty}^3 \mathcal{C}_{0,2}^{1/2}(q)
=-\frac{2}{q}j(q^2;q^5)A(-q)- \frac{J_{1}^4J_{4}J_{4,20}}{J_{2}^4}.
\end{equation}
\end{corollary}

\subsection{The string functions of negative admissible level} 
\label{subsection:strfuncnegadmlevel}

For negative admissible level $N<0$ as defined in Section \ref{subsection:admissiblerepresentations} we have $p^{\prime}<2p$ and hence the discriminant of Hecke-type double-sums in Proposition \ref{proposition:modStringFnHeckeForm} is negative $D = (p')^2-2pp' < 0$. Thus we can use Proposition \ref{proposition:modStringFnHeckeForm} coupled with Theorem \ref{theorem:negDisc} to obtain the representation of the string functions of negative admissible level $N$ in terms of false theta functions.

\begin{theorem}\label{theorem:genNegativeLevelExpansion}
Let $N = p'/p-2$ be a negative admissible level, that is, $p\geq 1$, $p'\geq 2$ are coprime integers and $p'<2p$. For $0\le \ell \le p^{\prime}-2$, and $m\in2\mathbb{Z}+\ell$ we have
\begin{align*}
(q)_{\infty}^{3}&\mathcal{C}_{m,\ell}^{N}(q)\\
&=- \frac{1}{2}\sum_{k=1}^{p-1} 
(-1)^{k}q^{(\frac{m-\ell}{2})k}q^{\binom{k}{2}}
\left ( j(-q^{p^{\prime}k+p(p^{\prime} -(\ell+1))};q^{2pp^{\prime}})
-q^{(1+\ell)k}
j(-q^{p^{\prime}k+p(p^{\prime}+\ell+1)};q^{2pp^{\prime}})\right) \\ 
&\qquad \times
\left ( \sum_{R\in\mathbb{Z}}\sg(R)
q^{-p^2NR^2+pR(m-kN)}
-q^{(k-p)(pN-m)}
\sum_{R\in\mathbb{Z}}\sg(R)
q^{-p^2NR^2+pR(m-(2p-k)N)}
\right ).
\end{align*}
\end{theorem}

We demonstrate the utility of Theorem \ref{theorem:genNegativeLevelExpansion} by giving two examples.  We give a new proof of the result on the $(-1/2)$-level string functions due to Schilling and Warnaar \cite[Section 6]{SW}.

\begin{corollary}\label{corollary:fractionalLevelMinus12} Let $(p,p^{\prime})=(2,3)$, $\ell \in \{0,1\}$ and $m \in 2\Z+ \ell$. We have
 \begin{equation*}
 \mathcal{C}_{m,\ell}^{-1/2}(q)
 =\frac{q^{\frac{1}{2}(m-\ell)}}{(q)_{\infty}^2}\sum_{i\ge 0}(-1)^iq^{\frac{1}{2}i(i+2m+1)}.
 \end{equation*}    
\end{corollary}

\noindent We also give an evaluation of the $(-2/3)$-level string functions.

\begin{corollary} \label{corollary:fractionalLevelMinus34} Let $(p,p^{\prime})=(3,4)$, $0 \leq \ell \leq 2$ and $m \in 2\Z+ \ell$.  We have
\begin{align*}
 \mathcal{C}_{m,\ell}^{-2/3}(q)
 &  =\frac{1}{2(q)_{\infty}^3J_{16}}\Big(  j(q^{1+\ell};q^{8})j(q^{10+2\ell};q^{16})\cdot q^{\frac{1}{2}(m-\ell)}
\sum_{r\in\mathbb{Z}}\sg(r)
q^{r(6r+3m+2)}\ \\
&\qquad \qquad +j(q^{5+\ell};q^{8})j(q^{2+2\ell};q^{16})\cdot q^{(2m-\ell)+3}
\sum_{r\in\mathbb{Z}}\sg(r)
q^{r(6r+3m+8)}\notag \\
&\qquad \qquad -j(q^{5+\ell};q^{8})j(q^{2+2\ell};q^{16})\cdot q^{(m-\ell)+1}
\sum_{r\in\mathbb{Z}}\sg(r) q^{r(6r+3m+4)}\notag \\
&\qquad \qquad -j(q^{1+\ell};q^{8})j(q^{10+2\ell};q^{16})\cdot q^{\frac{1}{2}(5m-\ell)+4}
\sum_{r\in\mathbb{Z}}\sg(r)
q^{r(6r+3m+10)}\Big ).\notag
\end{align*}
\end{corollary}

\subsection{Overview of the paper} 

In Section \ref{section:thetaAppellProperties}, we recall necessary properties for theta functions and Appell functions.  In Section \ref{section:twoTechnicalResults}, we prove useful identities for string functions and theta functions.  In Section \ref{section:genFractionalLevel} we give a proof of Proposition \ref{proposition:modStringFnHeckeForm} as well as useful cross-spin identities for $1/2$-level string functions.  In Section \ref{section:mixedmockprop}, we derive the modular properties of Theorem \ref{theorem:mockmodtransfstrfunc25l0}, and in Section \ref{section:proveidentities}, we apply the modular properties to prove two Hecke-type double-sum identities which will be critical to the proof of Theorem \ref{theorem:fractionalLevelPlus12}.  In Section \ref{section:fractionLevelPositive12}, we give proofs of Theorem \ref{theorem:fractionalLevelPlus12} and Corollary \ref{corollary:fractionalLevelPlus12-2ndA}. In Section \ref{section:generalNegativeLevel} we prove the general false theta expansion in Theorem \ref{theorem:genNegativeLevelExpansion}.  In Section \ref{section:fractionalLevelMinus12}, we give a new proof for the identity of Corollary \ref{corollary:fractionalLevelMinus12}.  In Section \ref{section:fractionalLevelMinus34} we prove Corollary \ref{corollary:fractionalLevelMinus34}.

\section{Properties of theta functions, Appell functions, and Hecke-type double-sums}\label{section:thetaAppellProperties}

Following from the definitions are the general identities:
\begin{subequations}
{\allowdisplaybreaks \begin{gather}
j(q^n x;q)=(-1)^nq^{-\binom{n}{2}}x^{-n}j(x;q), \ \ n\in\mathbb{Z},\label{equation:j-elliptic}\\
j(x;q)=j(q/x;q)=-xj(x^{-1};q)\label{equation:j-flip},\\
j(x;q)={J_1}j(x,qx,\dots,q^{n-1}x;q^n)/{J_n^n} \ \ {\text{if $n\ge 1$,}}\label{equation:1.10}\\
j(z;q)=\sum_{k=0}^{m-1}(-1)^k q^{\binom{k}{2}}z^k
j\big ((-1)^{m+1}q^{\binom{m}{2}+mk}z^m;q^{m^2}\big ),\label{equation:jsplit}\\
j(qx^3;q^3)+xj(q^2x^3;q^3)=j(-x;q)j(qx^2;q^2)/J_2={J_1j(x^2;q)}/{j(x;q)},\label{equation:quintuple}
\end{gather}}%
\end{subequations}
\noindent  where identity (\ref{equation:quintuple}) is the quintuple product identity.   For later use, we state the $m=2$ specialization of (\ref{equation:jsplit})
\begin{equation}
j(z;q)=j(-qz^2;q^4)-zj(-q^3z^2;q^4).\label{equation:jsplit-m2}
\end{equation}

When we use results from \cite{HM} to expand Hecke-type double-sums in terms of Appell functions and theta functions, we will need to use some elementary properties of Appell functions.
\begin{proposition}  For generic $x,z\in \mathbb{C}^*$
{\allowdisplaybreaks \begin{subequations}
\begin{gather}
m(x,z;q)=m(x,qz;q),\label{equation:mxqz-fnq-z}\\
m(x,z;q)=x^{-1}m(x^{-1},z^{-1};q),\label{equation:mxqz-flip}\\
m(qx,z;q)=1-xm(x,z;q).\label{equation:mxqz-fnq-x}
\end{gather}
\end{subequations}}
\end{proposition}
A straightforward induction argument yields a generalization of (\ref{equation:mxqz-fnq-z}):
 \begin{lemma} \label{lemma:appellUnwind}For $m\in \mathbb{Z}$, we have
  \begin{equation}
 m(q^n,-1;q)
 =\sum_{k=0}^{n-1}(-1)^kq^{nk-\binom{k+1}{2}}+(-1)^{n} q^{\binom{n}{2}}m(1,-1;q).
 \end{equation}
 \end{lemma}

We record the $n=2$ specialization of \cite[Theorem 3.5]{HM}.  We will need the following identity to write $1/2$-level string functions in terms of $\mu(q)$. 
\begin{corollary} \label{corollary:msplitn2zprime} For generic $x,z,z'\in \mathbb{C}^*$ 
{\allowdisplaybreaks \begin{align}
m(&x,z;q)=m(-qx^2,z';q^4 )-q^{-1}xm(-q^{-1}x^2,z';q^4) \label{equation:msplit2}\\
& + \frac{z'J_2^3}{j(xz;q)j(z';q^4)}\Big [
\frac{j(-qx^2zz';q^2)j(z^2/z';q^{4})}{j(-qx^2z';q^2)j(z;q^2)}-xz \frac{j(-q^2x^2zz';q^2)j(q^2z^2/z';q^{4})}{j(-qx^2z';q^2)j(qz;q^2)}\Big ].\notag
\end{align}}%
\end{corollary}
We follow with some useful specializations.
\begin{corollary}\label{corollary:msplitn2mock2ndmu} We have
\begin{equation}
m(1,-1;q)=2m(-q,-1;q^4 )
 -\frac{J_2^3}{\overline{J}_{0,1}\overline{J}_{0,4}J_{1,2}}\Big [
\frac{\overline{J}_{1,2}\overline{J}_{0,4}}{\overline{J}_{0,2}}
+  \frac{\overline{J}_{0,2}\overline{J}_{2,4}}{\overline{J}_{1,2}}
\Big ].\label{equation:msplit-n2}
\end{equation}
\end{corollary}

\begin{proof}[Proof of Corollary \ref{corollary:msplitn2mock2ndmu}]We further specialize (\ref{equation:msplit2}) to $x=1$, $z=z'=-1$.  This yields
\begin{align*}
m(1,-1;q)&=m(-q,-1;q^4 )-q^{-1}m(-q^{-1},-1;q^4)\\
&\qquad -\frac{J_2^3}{j(-1;q)j(-1;q^4)}\Big [
\frac{j(-q;q^2)j(-1;q^{4})}{j(q;q^2)j(-1;q^2)}+\frac{j(-q^2;q^2)j(-q^2;q^{4})}{j(q;q^2)j(-q;q^2)}\Big ],\notag
\end{align*}
and the result follows from (\ref{equation:mxqz-flip}).
\end{proof}

\begin{proposition}\label{proposition:genAppellToMockMu} For $m\in\mathbb{Z}$, we have
\begin{align*}
&m(-q^{1+2m},-1;q^4 )-q^{m-1}m(-q^{-1+2m},-1;q^4)\\
&\qquad =\sum_{k=0}^{m-1}(-1)^kq^{mk-\binom{k+1}{2}}
 +(-1)^{m} q^{\binom{m}{2}}\left ( \frac{1}{2}\mu(q)+\frac{1}{2}\frac{J_{2,4}^4}{J_1^3}\right ). 
\end{align*}
\end{proposition}
\begin{proof}[Proof of Proposition \ref{proposition:genAppellToMockMu}]
Specializing Corollary \ref{corollary:msplitn2zprime} with $(x,q,z,z')=(q^m,q^4,-1,-1)$ and use (\ref{equation:j-elliptic}) yields
{\allowdisplaybreaks \begin{align*}
m(q^{m},-1;q)&=m(-q^{1+2m},-1;q^4 )-q^{m-1}m(-q^{-1+2m},-1;q^4)\\
& \qquad - \frac{J_2^3}{j(-q^{m};q)\overline{J}_{0,4}}\Big [
\frac{j(-q^{1+2m};q^2)\overline{J}_{0,4}}{j(q^{1+2m};q^2)\overline{J}_{0,2}}+q^{m} \frac{j(-q^{2+2m};q^2)\overline{J}_{2,4}}{j(q^{1+2m};q^2)\overline{J}_{1,2}}\Big ]\\
&=m(-q^{1+2m},-1;q^4 )-q^{m-1}m(-q^{-1+2m},-1;q^4)\\
& \qquad - (-1)^{m}q^{\binom{m}{2}}
\frac{J_2^3}{\overline{J}_{0,1}\overline{J}_{0,4}}\Big [
\frac{j(-q;q^2)\overline{J}_{0,4}}{j(q;q^2)\overline{J}_{0,2}}+ \frac{j(-q^{2};q^2)\overline{J}_{2,4}}{j(q;q^2)\overline{J}_{1,2}}\Big ].
\end{align*}}%
Rearranging terms and employing Lemma \ref{lemma:appellUnwind} enables us to write
{\allowdisplaybreaks \begin{align*}
&m(-q^{1+2m},-1;q^4 )-q^{m-1}m(-q^{-1+2m},-1;q^4)\\
&\qquad =m(q^{m},-1;q)
+(-1)^{m}q^{\binom{m}{2}}
\frac{J_2^3}{\overline{J}_{0,1}\overline{J}_{0,4}J_{1,2}}\Big [
\frac{\overline{J}_{1,2}\overline{J}_{0,4}}{\overline{J}_{0,2}}+ \frac{\overline{J}_{0,2}\overline{J}_{2,4}}{\overline{J}_{1,2}}\Big ]\\
&\qquad =\sum_{k=0}^{m-1}(-1)^kq^{mk-\binom{k+1}{2}}+(-1)^{m} q^{\binom{m}{2}}m(1,-1;q)\\
&\qquad \qquad +(-1)^{m}q^{\binom{m}{2}}
\frac{J_2^3}{\overline{J}_{0,1}\overline{J}_{0,4}J_{1,2}}\Big [
\frac{\overline{J}_{1,2}\overline{J}_{0,4}}{\overline{J}_{0,2}}+ \frac{\overline{J}_{0,2}\overline{J}_{2,4}}{\overline{J}_{1,2}}\Big ].
\end{align*}}%
The result then follows from (\ref{equation:msplit-n2}) and (\ref{equation:muAppellForm}).
\end{proof}

Lastly, we recall some basic Hecke-type double-sum properties which will see much use.
\begin{proposition}\cite[Proposition 6.2]{HM}  \label{proposition:H7eq1.14}For $x,y\in\mathbb{C}^*$
\begin{equation}
f_{a,b,c}(x,y;q)=-\frac{q^{a+b+c}}{xy}f_{a,b,c}(q^{2a+b}/x,q^{2c+b}/y;q).\label{equation:H7eq1.14}
\end{equation}
\end{proposition}

\begin{proposition}\cite[Proposition 6.3]{HM}  \label{proposition:f-functionaleqn} For $x,y\in\mathbb{C}^*$ and $\ell, k \in \mathbb{Z}$
\begin{align}
f_{a,b,c}(x,y;q)&=(-x)^{\ell}(-y)^kq^{a\binom{\ell}{2}+b\ell k+c\binom{k}{2}}f_{a,b,c}(q^{a\ell+bk}x,q^{b\ell+ck}y;q) \notag\\
&\ \ \ \ +\sum_{m=0}^{\ell-1}(-x)^mq^{a\binom{m}{2}}j(q^{mb}y;q^c)+\sum_{m=0}^{k-1}(-y)^mq^{c\binom{m}{2}}j(q^{mb}x;q^a),\label{equation:Gen1}
\end{align}
where when $b<a$, we follow the usual summation convention (\ref{equation:sumconvention}).
\end{proposition}

\begin{corollary}\cite[Corollary 6.4]{HM} \label{corollary:fabc-funceqnspecial} We have two simple specializations:
\begin{align}
f_{a,b,c}(x,y;q) =&-yf_{a,b,c}(q^bx,q^cy;q)+j(x;q^a),\label{equation:fabc-fnq-1}\\
f_{a,b,c}(x,y;q) =&-xf_{a,b,c}(q^ax,q^by;q)+j(y;q^c).\label{equation:fabc-fnq-2}
\end{align}
\end{corollary}

Later in the paper, we will see that although the $1/2$-level string functions are expressed in terms of $f_{1,5,20}(x,y;q)$'s, it is much easier to evaluate them once they are in a more symmetric double-sum.  With this purpose in mind, we record the $(a,b,c)=(5,5,1)$ specialization of \cite[Theorem 1.4]{HM}:
\begin{lemma} \label{lemma:f551-expansion} We have
\begin{align}
f_{5,5,1}(x,y,q)&=h_{5,5,1}(x,y,-1,-1;q)\label{equation:f551}\\
&\ \ \ \ \ -\sum_{d=0}^{4}
q^{2d(d+1)}\frac{j\big (q^{4+4d}y;q^{5}\big )  j\big (-q^{16-4d}xy^{-1};q^{20}\big )J_{20}^3j\big (q^{14+4d}y^{-4};q^{20}\big )}
{\overline{J}_{0,4}\overline{J}_{0,20}j\big (q^{10}xy^{-5};q^{20})j(q^{4+4d}x^{-1}y;q^{20}\big )},\notag
\end{align}
where
\begin{equation}
h_{5,5,1}(x,y,z_1,z_0;q)
=j(x;q^5)m(-q^4x^{-1}y,z_1;q^4)
+j(y;q)m(-q^{10}xy^{-5},z_0;q^{20}).\label{equation:h551}
\end{equation}
\end{lemma}

\section{Technical results for theta functions}\label{section:twoTechnicalResults}

In this section, we obtain two technical results.

\begin{lemma} \label{lemma:generalSingleQuotienEll03} For $\ell\in\{0,3 \}$ and $m\in2\mathbb{Z}+\ell$, we have
\begin{align}
 j(q;q^5) \frac{1}{2} \frac{J_{2,4}^4}{J_1^3}
 &- q^{-1}\frac{J_{20}^3}{\overline{J}_{0,4}\overline{J}_{0,20}}
\sum_{d=0}^{4}q^{2d(d-1)}
\frac{j\big (q^{2+4d};q^{5}\big )  j\big (-q^{1+4d};q^{20}\big )j\big (q^{2+4d};q^{20}\big )}
{J_{1,20}j(q^{1+4d};q^{20}\big )}\notag \\
&\qquad  +\frac{J_{20}^3}{\overline{J}_{0,4}\overline{J}_{0,20}}
\sum_{d=0}^{4}q^{2d(d+1)}
 \frac{j\big (q^{4+4d};q^{5}\big )  j\big (-q^{3+4d};q^{20}\big )j\big (q^{14+4d};q^{20}\big )}
{J_{11,20}j(q^{3+4d};q^{20}\big )}\notag \\
&\qquad \qquad = 
\frac{1}{2}\cdot \frac{J_{1}^3J_{10}^{3}}{J_{4}J_{5}}\cdot \frac{1}{J_{1,10}J_{8,20}}.\label{equation:id1TechnicalTheta}
\end{align}
\end{lemma}

\begin{lemma}\label{lemma:generalSingleQuotienEll12} For $\ell\in\{1,2 \}$ and $m\in2\mathbb{Z}+\ell$, we have
\begin{align}
 j(q^{2};q^5)\frac{1}{2} \frac{J_{2,4}^4}{J_1^3}
 &-\frac{J_{20}^3}{\overline{J}_{0,4}\overline{J}_{0,20}}\sum_{d=0}^{4}q^{2d(d+1)}
\frac{j\big (q^{4+4d};q^{5}\big )  j\big (-q^{1+4d};q^{20}\big )j\big (q^{14+4d};q^{20}\big )}
{J_{7,20}j(q^{1+4d};q^{20}\big )}
\notag \\
&\qquad  -q^{-1}\frac{J_{20}^3}{\overline{J}_{0,4}\overline{J}_{0,20}}
\sum_{d=0}^{4}
q^{2d(d+1)} \frac{j\big (q^{1+4d};q^{5}\big )  j\big (-q^{3+4d};q^{20}\big )j\big (q^{6+4d};q^{20}\big )}
{J_{3,20}j(q^{3+4d};q^{20}\big )}\notag\\
&\qquad \qquad =- 
\frac{1}{2}\cdot \frac{J_{1}^3J_{10}^{3}}{J_{4}J_{5}}\cdot \frac{1}{J_{3,10}J_{4,20}}. \label{equation:id2TechnicalTheta}
\end{align}
\end{lemma}

\begin{proof} [Proofs of Lemmas \ref{lemma:generalSingleQuotienEll03} and \ref{lemma:generalSingleQuotienEll12}]

We use Frye and Garvan's Maple packages {\em qseries} and {\em thetaids} to prove both theta function identities \cite{FG}.   As an example, we sketch how to prove (\ref{equation:id1TechnicalTheta}).  We first normalize (\ref{equation:id1TechnicalTheta}) to obtain the equivalent identity
\begin{equation}
g({\tau}):=\theta (\tau)-\sum_{d=0}^{4}f_{d}(\tau)
+\sum_{d=0}^{4}h_{d}(\tau)-1=0,\label{equation:finalTheta-id2-normal}
\end{equation}
where
\begin{equation*}
\theta (\tau):=\frac{1}{\Psi_{2}(q)}j(q;q^5) \frac{1}{2} \frac{J_{2,4}^4}{J_1^3},
\end{equation*}
{\allowdisplaybreaks \begin{gather*}
f_d(\tau):=q^{-1}\frac{\Psi_1(q)}{\Psi_2(q)}
\cdot q^{2d(d-1)}\frac{j\big (q^{2+4d};q^{5}\big )  j\big (-q^{1+4d};q^{20}\big )j\big (q^{2+4d};q^{20}\big )}
{J_{1,20}j(q^{1+4d};q^{20}\big )}, \\
h_d(\tau):=\frac{\Psi_1(q)}{\Psi_2(q)}
\cdot  q^{2d(d+1)}\frac{j\big (q^{4+4d};q^{5}\big )  j\big (-q^{3+4d};q^{20}\big )j\big (q^{14+4d};q^{20}\big )}
{J_{11,20}j(q^{3+4d};q^{20}\big )},
\end{gather*}}%
and
\begin{equation*}
\Psi_1(q):=\frac{1}{4}\frac{J_{20}^4J_{4}}{J_{8}^2J_{40}^2}, \ 
\Psi_2(q):=\frac{1}{2}\frac{J_1^3J_{10}^3}{J_{4}J_{5}}\frac{1}{J_{1,10}J_{8,20}}.
\end{equation*}

Firstly, we use  \cite[Theorem 18]{Rob} to verify that each $\theta(\tau)$, $f_{d}(\tau)$, and $h_d(\tau)$ is a modular function on $\Gamma_{1}(40)$ for $1\le j\le 5$.  Secondly, we use \cite[Corollary 4]{CKP} to find a set $\mathcal{S}_{40}$ of inequivalent cusps for $\Gamma_{1}(40)$.  We determine the fan width of each cusp.  Then, we use \cite[Lemma 3.2]{Biag} to determine the invariant order of each modular function at each of the cusps of $\Gamma_{1}(40)$.  Next, we use the Valence Formula \cite[p. 98]{Rank} to calculate the number of terms needed to verify so that we can confirm identity (\ref{equation:finalTheta-id2-normal}).  We calculate that 
\begin{equation*}
B:=\sum_{\substack{s\in\mathcal{S}_{40}\\ s\ne i\infty}}\textup{min}(\{ \textup{ORD}(f_{j},s,\Gamma_{1}(40)):1\le j\le n\}\cup\{0\}),
\end{equation*}
where $\textup{ORD}(f,\zeta,\Gamma):=\kappa(\zeta,\Gamma)\textup{ORD}(f,\zeta)$, with $\kappa(\zeta,\Gamma)$ denoting the fan width of the cusp $\zeta$ and $\textup{ORD}(f,\zeta)$ denoting the invariant order.  More details can be found in \cite[p. 91]{Rank}.   We find that $B=-58$.  From the Valence Formula \cite[Corollary 2.5]{FG} we know that (\ref{equation:finalTheta-id2-normal}) is true if and only if 
\begin{equation*}
\textup{ORD}(g(\tau), i\infty,\Gamma_{1}(40))>-B
\end{equation*}
In the final step, we verify identity (\ref{equation:finalTheta-id2-normal}) out through $\mathcal{O}(q^{60}).$
\end{proof}


\section{The Hecke-type double-sum form of the string functions}\label{section:genFractionalLevel}

In this section we prove Proposition \ref{proposition:modStringFnHeckeForm}.  We also obtain an immediate corollary, which we will find to be very useful, which gives examples of relations between string functions of different spins.  However, one can also obtain the identities directly from \cite[$(3.8)$]{SW}.

\begin{corollary} \label{corollary:crossspinidentities} We have
\begin{equation}  \label{eq:crossspiniden0211}
(q)_{\infty}^3 \mathcal{C}_{1,1}^{1/2}(q)+q(q)_{\infty}^3 \mathcal{C}_{0,2}^{1/2}(q)
=j(q^2;q^5)
\end{equation}
and
\begin{equation}  \label{eq:crossspiniden0013}
(q)_{\infty}^3 \mathcal{C}_{0,0}^{1/2}(q)+q(q)_{\infty}^3 \mathcal{C}_{1,3}^{1/2}(q)
=j(q;q^5).
\end{equation}
\end{corollary}

\begin{proof}[Proof of Proposition \ref{proposition:modStringFnHeckeForm}]
Recall that $m\in 2\mathbb{Z}+\ell$.   Let us rewrite (\ref{equation:fractional-string}) piece-wise in terms of Hecke-type double-sums (\ref{equation:fabc-def2}).  For the first summand, it is straightforward to write
\begin{align*}
 \left \{ \sum_{\substack{i\ge 0 \\ j\ge 0}}-\sum_{\substack{i< 0 \\ j< 0}}\right \}
  (-1)^{i}q^{\frac{1}{2}i(i+m)+p^{\prime}j(pj+i)+\frac{1}{2}(\ell+1)(2pj+i)}=f_{1,p^{\prime},2pp^{\prime}}(q^{1+\frac{m+\ell}{2}},-q^{p(1+p^{\prime}+\ell)};q).
\end{align*}
For the second summand we have
{\allowdisplaybreaks \begin{align*}
& \left \{  \sum_{\substack{i\ge 0 \\ j> 0}}-\sum_{\substack{i< 0 \\ j\le 0}}\right \}
 (-1)^{i}q^{\frac{1}{2}i(i+m)+p^{\prime}j(pj+i)-\frac{1}{2}(\ell+1)(2pj+i)}\\
& = \left \{ \sum_{\substack{i\ge 0 \\ j\ge 0}}-\sum_{\substack{i< 0 \\ j< 0}}\right \}
 (-1)^{i}q^{\frac{1}{2}i(i+m)+p^{\prime}j(pj+i)-\frac{1}{2}(\ell+1)(2pj+i)}\\
 &\qquad -\sum_{\substack{i\ge 0 \\ j= 0}} (-1)^{i}q^{\frac{1}{2}i(i+m)+p^{\prime}j(pj+i)-\frac{1}{2}(\ell+1)(2pj+i)}
  -\sum_{\substack{i< 0 \\ j=0}}
 (-1)^{i}q^{\frac{1}{2}i(i+m)+p^{\prime}j(pj+i)-\frac{1}{2}(\ell+1)(2pj+i)}\\
&= f_{1,p^{\prime},2pp^{\prime}}(q^{\frac{m-\ell}{2}},-q^{p(p^{\prime}-(\ell+1))};q)
-\sum_{i\ge 0 } (-1)^{i}q^{\frac{1}{2}i(i+m)-\frac{1}{2}(\ell+1)i}
  -\sum_{i< 0 }
 (-1)^{i}q^{\frac{1}{2}i(i+m)-\frac{1}{2}(\ell+1)i}\\
&=  f_{1,p^{\prime},2pp^{\prime}}(q^{\frac{m-\ell}{2}},-q^{p(p^{\prime}-(\ell+1))};q)
  -\sum_{i\in\mathbb{Z} } (-1)^{i}q^{\frac{1}{2}i(i+m)-\frac{1}{2}(\ell+1)i}\\
&=  f_{1,p^{\prime},2pp^{\prime}}(q^{\frac{m-\ell}{2}},-q^{p(p^{\prime}-(\ell+1))};q)
 -j(q^{\frac{1}{2}(m-\ell)};q)\\
 &=  f_{1,p^{\prime},2pp^{\prime}}(q^{\frac{m-\ell}{2}},-q^{p(p^{\prime}-(\ell+1))};q),
\end{align*}}%
because $j(q^{n};q)=0$ for $n\in \mathbb{Z}$.
Thus we arrive at the desired form.
\end{proof}

 \begin{proof}[Proof of Corollary \ref{corollary:crossspinidentities}]
 We record certain specializations of functional equation properties found in \cite{HM}.  From Proposition \ref{proposition:H7eq1.14}, we have  
\begin{equation}
f_{1,5,20}(x,y;q)=-\frac{q^{26}}{xy}f_{1,5,20}(q^{7}/x,q^{45}/y;q),\label{equation:fabc-flip}
\end{equation}
and from Corollary \ref{corollary:fabc-funceqnspecial}, we have
\begin{align}
f_{1,5,20}(x,y;q) =&-yf_{1,5,20}(q^5x,q^{20}y;q)+j(x;q),\label{equation:fabc-fnq1}\\
f_{1,5,20}(x,y;q) =&-xf_{1,5,20}(qx,q^5y;q)+j(y;q^{20}).\label{equation:fabc-fnq2}
\end{align}

 We prove (\ref{eq:crossspiniden0211}).  Proposition \ref{proposition:modStringFnHeckeForm} gives
\begin{equation*}
(q)_{\infty}^3 \mathcal{C}_{1,1}^{1/2}(q)
= f_{1,5,20}(q^{2},-q^{14};q)
 -f_{1,5,20}(1,-q^{6};q)
\end{equation*}
and
\begin{equation*}
q(q)_{\infty}^3 \mathcal{C}_{0,2}^{1/2}(q)
=q f_{1,5,20}(q^{2},-q^{16};q)
 -qf_{1,5,20}(q^{-1},-q^{4};q).
\end{equation*}
We begin by using (\ref{equation:fabc-fnq2}) to obtain
\begin{align*}
(q)_{\infty}^3 \mathcal{C}_{1,1}^{1/2}(q)
&=-q^2f_{1,5,20}(q^3,-q^{19};q)+j(-q^{14};q^{20})+f_{1,5,20}(q,-q^{11};q)-j(-q^{6};q^{20})\\
&=-q^2f_{1,5,20}(q^3,-q^{19};q)+f_{1,5,20}(q,-q^{11};q).
\end{align*}
Using (\ref{equation:fabc-fnq2}) again yields
\begin{align*}
(q)_{\infty}^3 \mathcal{C}_{1,1}^{1/2}(q)
&=q^5f_{1,5,20}(q^4,-q^{24};q)-q^2j(-q^{19};q^{20})-qf_{1,5,20}(q^2,-q^{16};q)+j(-q^{11};q^{20})\\
&=q^5f_{1,5,20}(q^4,-q^{24};q)-qf_{1,5,20}(q^2,-q^{16};q)+j(q^2;q^5),
\end{align*}
where in the last equality we used the fact that
\begin{equation*}
j(x;q)=j(-qx^2;q^4)-xj(-q^3x^2;q^4).
\end{equation*}
Applying (\ref{equation:fabc-fnq1}) gives
\begin{align*}
(q)_{\infty}^3 \mathcal{C}_{1,1}^{1/2}(q)
&=qf_{1,5,20}(q^{-1},-q^{4};q)-qf_{1,5,20}(q^2,-q^{16};q)+j(q^2;q^5)\\
&=-q(q)_{\infty}^3 \mathcal{C}_{0,2}^{1/2}(q)+j(q^2;q^5).\qedhere
\end{align*}

To prove (\ref{eq:crossspiniden0013}), we first note that Proposition \ref{proposition:modStringFnHeckeForm} yields
\begin{equation*}
(q)_{\infty}^3 \mathcal{C}_{0,0}^{1/2}(q)
= f_{1,5,20}(q,-q^{12};q)
 -f_{1,5,20}(1,-q^{8};q)
\end{equation*}
and
\begin{equation*}
q(q)_{\infty}^3 \mathcal{C}_{1,3}^{1/2}(q)
=q f_{1,5,20}(q^{3},-q^{18};q)
 -qf_{1,5,20}(q^{-1},-q^{2};q).
\end{equation*}
Using (\ref{equation:fabc-fnq2}) provides
\begin{align*}
(q)_{\infty}^3 \mathcal{C}_{0,0}^{1/2}(q)
&=-qf_{1,5,20}(q^2,-q^{17};q)+j(-q^{12};q^{20})+f_{1,5,20}(q,-q^{13};q)-j(-q^{8};q^{20})\\
&=-qf_{1,5,20}(q^2,-q^{17};q)+f_{1,5,20}(q,-q^{13};q).
\end{align*}
Using (\ref{equation:fabc-fnq2}) again yields
\begin{align*}
(q)_{\infty}^3 \mathcal{C}_{0,0}^{1/2}(q)
&=q^3f_{1,5,20}(q^3,-q^{22};q)-qj(-q^{17};q^{20})-qf_{1,5,20}(q^2,-q^{18};q)+j(-q^{13};q^{20})\\
&=q^3f_{1,5,20}(q^3,-q^{22};q)-qf_{1,5,20}(q^2,-q^{18};q)+j(q;q^5).
\end{align*}
Once again, applying (\ref{equation:fabc-fnq2}) produces
\begin{align*}
(q)_{\infty}^3 &\mathcal{C}_{0,0}^{1/2}(q)\\
&=-q^6f_{1,5,20}(q^4,-q^{27};q)+q^3j(-q^{22};q^{20})+q^3f_{1,5,20}(q^3,-q^{23};q)-qj(-q^{18};q^{20})+j(q;q^5)\\
&=-q^6f_{1,5,20}(q^4,-q^{27};q)+q^3f_{1,5,20}(q^3,-q^{23};q)+j(q;q^5).
\end{align*}
where the last line follows from $j(-qx;q)=-x^{-1}j(-x;q)$.  Applying (\ref{equation:fabc-flip}) gives
\begin{align*}
(q)_{\infty}^3 \mathcal{C}_{0,0}^{1/2}(q)
&=-qf_{1,5,20}(q^{3},-q^{18};q)+q^3f_{1,5,20}(q^4,-q^{22};q)+j(q;q^5)\\
&=-qf_{1,5,20}(q^{3},-q^{18};q)+qf_{1,5,20}(q^{-1},-q^{2};q)-qj(-q^{-1};q)+j(q;q^5)\\
&=-qf_{1,5,20}(q^{3},-q^{18};q)+qf_{1,5,20}(q^{-1},-q^{2};q)+j(q;q^5).
\end{align*}
where in the penultimate line we have used (\ref{equation:fabc-fnq1}).  The result follows.\qedhere
\end{proof}


\section{The mixed mock modular transformations of $1/2$-level string functions} \label{section:mixedmockprop}

To present further results, we need to introduce the notation of Zwegers' thesis \cite[Chapter 2]{Zw02}. Let $A$ be a symmetric $r \times r$-matrix with integer coefficients, which is non-degenerate. We consider the quadratic form $Q : \C^r \rightarrow \C$, $Q(x) = \frac{1}{2} \langle x, Ax \rangle$ and the associated bilinear form $B(x, y) = \langle x, Ay\rangle = Q(x + y) - Q(x) - Q(y)$.

We assume that the largest dimension of a linear subspace of $\R^r$ on which $Q$ is negative definite is equal to $1$. Then the
set of vectors $c \in \R^r$ with $Q(c) < 0$ has two components. If $B(c_1,c_2) < 0$ then $c_1$ and $c_2$ belong to the same component, while if $B(c_1, c_2)> 0$  then $c_1$ and $c_2$ belong to opposite components. Let $C_Q$ be one of the two components. If $c_0$ is a vector in that component, then $C_Q$ is given by:
\begin{equation*}
C_Q := \{c \in \R^r \ | \ Q(c)<0, \ B(c,c_0)<0 \}.   
\end{equation*}
Let $c_1,c_2 \in C_Q$. We define the indefinite  $\vartheta$-function of $Q$ with characteristics $a,b \in \R^r$, with respect to parameters $(c_1, c_2)$ by
\begin{multline*}
\vartheta_{a,b}(\tau) = \vartheta_{a,b}^{Q,(c_1,c_2)}(\tau):= \\
\sum_{\upsilon \in a + \Z^r} \left(E\left(\frac{B(c_1,\upsilon)}{\sqrt{-Q(c_1)}} \sqrt{\im(\tau)} \right) - E\left(\frac{B(c_2,\upsilon)}{\sqrt{-Q(c_2)}} \sqrt{\im(\tau)} \right) \right) e^{2\pi i B(\upsilon,b)} q^{Q(\upsilon)},
\end{multline*}
where
\begin{equation*}
E(z) := \sgn(z)(1-\beta(z^2))
\end{equation*}
and
\begin{equation*}
\beta(x) := \int_{x}^{\infty} u^{-\frac{1}{2}} e^{\pi u} du, \ x \geq 0.
\end{equation*}
Let us define the holomorphic part of indefinite $\vartheta$-function
\begin{equation*}
\vartheta_{a,b}^{\operatorname{hol}}(\tau) = \vartheta_{a,b}^{Q,(c_1,c_2),\operatorname{hol}}(\tau) := \sum_{\upsilon \in a + \Z^r} (\sgn(B(c_1,\upsilon))-\sgn(B(c_2,\upsilon))) e^{2\pi i B(\upsilon,b)} q^{Q(\upsilon)},
\end{equation*}
and
the non-holomorphic part
\begin{equation*}
\vartheta_{a,b}^{\operatorname{nhol}}(\tau) := \vartheta_{a,b}(\tau) - \vartheta_{a,b}^{\operatorname{hol}}(\tau). 
\end{equation*}
Also let us introduce the non-holomorphic Mordell integral \cite[Proposition 1.9]{Zw02}
\begin{equation*}
R_{a,b}(\tau) := -i \int_{-\overline{\tau}}^{\infty} \frac{g_{a,-b}(z)}{\sqrt{-i(z+\tau)}} dz,
\end{equation*}
where we define a unary theta function as
\begin{equation*}
g_{a,b}(\tau) := \sum_{\zeta \in a+\Z}  \zeta e^{2\pi i \zeta b} q^{\frac{\zeta^2}{2}}.
\end{equation*}

In the next proposition we present the string functions $C_{0,0}^{(2,5)}(\tau)$ and $C_{0,2}^{(2,5)}(\tau)$ as the holomorphic parts of indefinite $\vartheta$-functions.
\begin{proposition} \label{proposition:RHSshortidentitiesasholpart}We have
\begin{equation*}
\begin{pmatrix}
\vartheta_{\begin{psmallmatrix}
    0\\
    1 \slash 10
\end{psmallmatrix},\begin{psmallmatrix}
    0\\
    1 \slash 2
\end{psmallmatrix}}\\
\vartheta_{\begin{psmallmatrix}
    0\\
    3 \slash 10
\end{psmallmatrix},\begin{psmallmatrix}
    0\\
    1 \slash 2
\end{psmallmatrix}}
\end{pmatrix}(\tau)
= \eta(\tau)^3 \cdot \begin{pmatrix}
  C^{(2,5)}_{0,0}\\
  C^{(2,5)}_{0,2}
\end{pmatrix}(\tau) 
+ R_{\frac{1}{4},0}(4\tau) \cdot \begin{pmatrix}
 \mathcal{J}_{1,5}\\
 \mathcal{J}_{2,5}
\end{pmatrix}(\tau),
\end{equation*}    
where the quadratic form of the indefinite $\vartheta$-functions is
\begin{equation*}
Q_{L}(x) :=  \frac{1}{2}x_1^2 + 5x_1x_2 + 10x_2^2 
\  \ \ \text{for} \ \ \ x=\begin{pmatrix}
 x_1\\
 x_2
\end{pmatrix} \in \R^2
\end{equation*}
and the parameters are
\begin{equation*}
c_1^{L} = \begin{pmatrix}
-4\\
1
\end{pmatrix}, \ \ \ c_2^{L} = \begin{pmatrix}
-5\\
1
\end{pmatrix}.   
\end{equation*}
\end{proposition}

\begin{proof}[Proof of Proposition \ref{proposition:RHSshortidentitiesasholpart}]
In this proof we use the notation $c_1=c_1^{L}$, $c_2=c_2^{L}$. It is easy to check that $Q(c_1)<0$ , $Q(c_2) < 0$ and $B(c_1,c_2)<0$. Also we can calculate
\begin{align*}
B(c_1,\upsilon) = \upsilon_1, \ 
B(c_2,\upsilon) = -5\upsilon_2,
\end{align*}
for $\upsilon = \begin{pmatrix}
    \upsilon_1\\
    \upsilon_2
\end{pmatrix} \in \R^2$. So we can calculate by definition
\begin{align*}
2q^{\frac{1}{10}}f_{1,5,20}(q,-q^{12};q) &= \vartheta_{\begin{psmallmatrix}
    0\\
    1 \slash 10
\end{psmallmatrix},\begin{psmallmatrix}
    0\\
    1 \slash 2
\end{psmallmatrix}}^{\operatorname{hol}}(\tau),\\
2q^{\frac{1}{10}}f_{1,5,20}(1,-q^{8};q) &= \vartheta_{\begin{psmallmatrix}
    0\\
    -1 \slash 10
\end{psmallmatrix},\begin{psmallmatrix}
    0\\
    1 \slash 2
\end{psmallmatrix}}^{\operatorname{hol}}(\tau),\\
2q^{\frac{9}{10}}f_{1,5,20}(q^2,-q^{16};q) &= \vartheta_{\begin{psmallmatrix}
    0\\
    3 \slash 10
\end{psmallmatrix},\begin{psmallmatrix}
    0\\
    1 \slash 2
\end{psmallmatrix}}^{\operatorname{hol}}(\tau),\\
2q^{\frac{9}{10}}f_{1,5,20}(q^{-1},-q^{4};q) &= \vartheta_{\begin{psmallmatrix}
    0\\
    -3 \slash 10
\end{psmallmatrix},\begin{psmallmatrix}
    0\\
    1 \slash 2
\end{psmallmatrix}}^{\operatorname{hol}}(\tau).
\end{align*}
Now let us consider the non-holomorphic parts. We can write
\begin{align}
\begin{split} \label{equation:nonholpartLHS1}
\vartheta_{\begin{psmallmatrix}
    0\\
    1 \slash 10
\end{psmallmatrix},\begin{psmallmatrix}
    0\\
    1 \slash 2
\end{psmallmatrix}}^{\operatorname{nhol}}  =  &-\sum_{\upsilon \in \begin{psmallmatrix}
    0\\
    1 \slash 10
\end{psmallmatrix} + \mathbb{Z}^2} \sgn(B(c_1,\upsilon)) \beta\left(-\frac{B(c_1,\upsilon)^2}{Q(c_1)}\im(\tau)\right) e^{2\pi i B\left(\upsilon,\begin{psmallmatrix}
    0\\
    1 \slash 2
\end{psmallmatrix}\right)} q^{Q(\upsilon)} \\
&+\sum_{\upsilon \in \begin{psmallmatrix}
    0\\
    1 \slash 10
\end{psmallmatrix}+ \mathbb{Z}^2} \sgn(B(c_2,\upsilon)) \beta\left(-\frac{B(c_2,\upsilon)^2}{Q(c_2)}\im(\tau)\right) e^{2\pi i B\left(\upsilon,\begin{psmallmatrix}
    0\\
    1 \slash 2
\end{psmallmatrix}\right)} q^{Q(\upsilon)}.
\end{split}
\end{align}
Using \cite[Proposition 4.3]{Zw02} we see for the first summand in \eqref{equation:nonholpartLHS1}
\begin{align*}
P_0 = \Bigg\{\begin{pmatrix}
    0\\
    1 \slash 10
\end{pmatrix},
\begin{pmatrix}
    -1\\
    1 \slash 10
\end{pmatrix},
\begin{pmatrix}
    -2\\
    1 \slash 10
\end{pmatrix},
\begin{pmatrix}
    -3\\
    1 \slash 10
\end{pmatrix}\Bigg\}, \ 
\langle c_1 \rangle^{\bot}_{\mathbb{Z}} = \Bigg\{ \begin{pmatrix}
0\\
\zeta_2
\end{pmatrix} \ \Bigg| \ \zeta_2 \in \mathbb{Z} \Bigg\},
\end{align*}
and we derive
\begin{multline*}
-\sum_{\upsilon \in \begin{psmallmatrix}
    0\\
    1 \slash 10
\end{psmallmatrix} + \mathbb{Z}^2} \sgn(B(c_1,\upsilon)) \beta\left(-\frac{B(c_1,\upsilon)^2}{Q(c_1)}\im(\tau)\right) e^{2\pi i B\left(\upsilon,\begin{psmallmatrix}
    0\\
    1 \slash 2
\end{psmallmatrix}\right)} q^{Q(\upsilon)} 
\\ =R_{0,2}(4\tau) \cdot \sum_{\zeta \in \frac{1}{10}+\mathbb{Z}} q^{10\zeta^2}
+R_{\frac{1}{4},2}(4\tau) \cdot \sum_{\zeta \in -\frac{3}{20}+\mathbb{Z}} q^{10\zeta^2}\\
+R_{\frac{1}{2},2}(4\tau) \cdot \sum_{\zeta \in -\frac{2}{5}+\mathbb{Z}} q^{10\zeta^2}
+R_{\frac{3}{4},2}(4\tau) \cdot \sum_{\zeta \in -\frac{13}{20}+\mathbb{Z}} q^{10\zeta^2}.  
\end{multline*}
Let us rewrite it using \cite[Proposition 4.2]{Zw02}
\begin{equation*}
R_{a,b}(\tau) =: i q^{-\frac{1}{2}(a-\frac{1}{2})^2} \cdot e^{-2\pi i (a-\frac{1}{2})b} \cdot R\left(\left(a-\frac{1}{2}\right)\tau+b+\frac{1}{2};\tau\right).
\end{equation*}
By applying \cite[Proposition 1.9]{Zw02} we have
\begin{align*}
R_{0,2}(4\tau) &= iq^{-\frac{1}{2}}R\left(-2\tau+\frac{5}{2};4\tau\right) = 1,\\
R_{\frac{1}{4},2}(4\tau) &= -R_{\frac{3}{4},2}(4\tau) = iq^{-\frac{1}{8}}R\left(-\tau+\frac{5}{2};4\tau\right)=R_{\frac{1}{4},0}(4\tau)  = -R_{\frac{3}{4},0}(4\tau),\\
R_{\frac{1}{2},2}(4\tau) &= i R\left(\frac{5}{2};4\tau\right) = 0.
\end{align*}
Using 
\begin{equation*}
\overline{J}_{7,20} - q\overline{J}_{3,20} = J_{1,5}    
\end{equation*}
we arrive at
\begin{multline*}
-\sum_{\upsilon \in \begin{psmallmatrix}
    0\\
    1 \slash 10
\end{psmallmatrix} + \mathbb{Z}^2} \sgn(B(c_1,\upsilon)) \beta\left(-\frac{B(c_1,\upsilon)^2}{Q(c_1)}\im(\tau)\right) e^{2\pi i B\left(\upsilon,\begin{psmallmatrix}
    0\\
    1 \slash 2
\end{psmallmatrix}\right)} q^{Q(\upsilon)} 
\\ = \overline{\mathcal{J}}_{8,20}(\tau) + R_{\frac{1}{4},0}(4\tau) \cdot \mathcal{J}_{1,5}(\tau).  
\end{multline*} 
Using \cite[Proposition 4.3]{Zw02} we see for the second summand in \eqref{equation:nonholpartLHS1}
\begin{align*}
P_0 = \Bigg\{\begin{pmatrix}
    0\\
    1 \slash 10
\end{pmatrix}\Bigg\}, \ 
\langle c_2 \rangle^{\bot}_{\mathbb{Z}} = \Bigg\{ \begin{pmatrix}
\zeta_1\\
0
\end{pmatrix} \ \Bigg| \ \zeta_1 \in \mathbb{Z} \Bigg\},
\end{align*}
and we derive
\begin{multline*}
-\sum_{\upsilon \in \begin{psmallmatrix}
    0\\
    1 \slash 10
\end{psmallmatrix} + \mathbb{Z}^2} \sgn(B(c_2,\upsilon)) \beta\left(-\frac{B(c_2,\upsilon)^2}{Q(c_2)}\im(\tau)\right) e^{2\pi i B\left(\upsilon,\begin{psmallmatrix}
    0\\
    1 \slash 2
\end{psmallmatrix}\right)} q^{Q(\upsilon)} 
\\ =-R_{\frac{1}{10},\frac{5}{2}}(5\tau) \cdot \sum_{\zeta \in \frac{1}{2}+\mathbb{Z}} q^{\frac{1}{2}\zeta^2} e^{\pi i \zeta} = 0.  
\end{multline*}
Finally we have 
\begin{equation*}
\vartheta_{\begin{psmallmatrix}
    0\\
    1 \slash 10
\end{psmallmatrix},\begin{psmallmatrix}
    0\\
    1 \slash 2
\end{psmallmatrix}}^{\operatorname{nhol}}(\tau)  = \overline{\mathcal{J}}_{8,20}(\tau) + R_{\frac{1}{4},0}(4\tau) \cdot \mathcal{J}_{1,5}(\tau).
\end{equation*}
Similarly we can find
\begin{align*}
\vartheta_{\begin{psmallmatrix}
    0\\
    -1 \slash 10
\end{psmallmatrix},\begin{psmallmatrix}
    0\\
    1 \slash 2
\end{psmallmatrix}}^{\operatorname{nhol}}(\tau)  &= \overline{\mathcal{J}}_{8,20}(\tau) - R_{\frac{1}{4},0}(4\tau) \cdot \mathcal{J}_{1,5}(\tau),\\
\vartheta_{\begin{psmallmatrix}
    0\\
    3 \slash 10
\end{psmallmatrix},\begin{psmallmatrix}
    0\\
    1 \slash 2
\end{psmallmatrix}}^{\operatorname{nhol}}(\tau)  &= \overline{\mathcal{J}}_{4,20}(\tau) + R_{\frac{1}{4},0}(4\tau) \cdot \mathcal{J}_{2,5}(\tau),\\
\vartheta_{\begin{psmallmatrix}
    0\\
    -3 \slash 10
\end{psmallmatrix},\begin{psmallmatrix}
    0\\
    1 \slash 2
\end{psmallmatrix}}^{\operatorname{nhol}}(\tau)  &= \overline{\mathcal{J}}_{4,20}(\tau) - R_{\frac{1}{4},0}(4\tau) \cdot \mathcal{J}_{2,5}(\tau).
\end{align*}
After summing up holomorphic and non-holomorphic parts we obtain the desired result.
\end{proof}

\begin{remark} \label{remark:LHStransfproperties} Note that from \cite[Corollary 2.9]{Zw02} we have the transformation properties on $\Gamma(1)$
\begin{equation*}
\begin{pmatrix}
\vartheta_{\begin{psmallmatrix}
    0\\
    1 \slash 10
\end{psmallmatrix},\begin{psmallmatrix}
    0\\
    1 \slash 2
\end{psmallmatrix}}\\
\vartheta_{\begin{psmallmatrix}
    0\\
    3 \slash 10
\end{psmallmatrix},\begin{psmallmatrix}
    0\\
    1 \slash 2
\end{psmallmatrix}}
\end{pmatrix}(\tau+1) = \begin{pmatrix}
\zeta_{10} & 0\\
0 & \zeta_{10}^{-1} \\
\end{pmatrix}
\begin{pmatrix}
\vartheta_{\begin{psmallmatrix}
    0\\
    1 \slash 10
\end{psmallmatrix},\begin{psmallmatrix}
    0\\
    1 \slash 2
\end{psmallmatrix}}\\
\vartheta_{\begin{psmallmatrix}
    0\\
    3 \slash 10
\end{psmallmatrix},\begin{psmallmatrix}
    0\\
    1 \slash 2
\end{psmallmatrix}}
\end{pmatrix}(\tau)    
\end{equation*}
and
\begin{equation*}
\begin{pmatrix}
\vartheta_{\begin{psmallmatrix}
    0\\
    1 \slash 10
\end{psmallmatrix},\begin{psmallmatrix}
    0\\
    1 \slash 2
\end{psmallmatrix}}\\
\vartheta_{\begin{psmallmatrix}
    0\\
    3 \slash 10
\end{psmallmatrix},\begin{psmallmatrix}
    0\\
    1 \slash 2
\end{psmallmatrix}}
\end{pmatrix}\left(-\frac{1}{\tau}\right) = (-i\tau) \cdot \frac{2}{\sqrt{5}}\begin{pmatrix}
\sin\left(\frac{2\pi}{5}\right) & -\sin\left(\frac{\pi}{5}\right) \\
-\sin\left(\frac{\pi}{5}\right)  & -\sin\left(\frac{2\pi}{5}\right)  \\
\end{pmatrix} \begin{pmatrix}
\vartheta_{\begin{psmallmatrix}
    0\\
    1 \slash 10
\end{psmallmatrix},\begin{psmallmatrix}
    0\\
    1 \slash 2
\end{psmallmatrix}}\\
\vartheta_{\begin{psmallmatrix}
    0\\
    3 \slash 10
\end{psmallmatrix},\begin{psmallmatrix}
    0\\
    1 \slash 2
\end{psmallmatrix}}
\end{pmatrix}(\tau).
\end{equation*}
\end{remark}

Now we are able to give a proof of Theorem \ref{theorem:mockmodtransfstrfunc25l0}.

\begin{proof}[Proof of Theorem \ref{theorem:mockmodtransfstrfunc25l0}]
From \cite[Proposition 1.3]{Zw02} and \cite[p. 6]{Zw19} one can derive the transformation properties of theta functions
\begin{equation*}
\begin{pmatrix}
\mathcal{J}_{1,5}\\
\mathcal{J}_{2,5}
\end{pmatrix}(\tau+1)  = 
\begin{pmatrix}
\zeta_{40}^{9} & 0\\
0 & \zeta_{40}\\
\end{pmatrix} \begin{pmatrix}
\mathcal{J}_{1,5}\\
\mathcal{J}_{2,5}
\end{pmatrix}(\tau)
\end{equation*}
and
\begin{equation*}
\begin{pmatrix}
\mathcal{J}_{1,5}\\
\mathcal{J}_{2,5}
\end{pmatrix}\left(-\frac{1}{\tau}\right)  =  \sqrt{-i\tau} \cdot \frac{2}{\sqrt{5}} 
\begin{pmatrix}
\sin\left(\frac{2\pi}{5}\right) & -\sin\left(\frac{\pi}{5}\right) \\
-\sin\left(\frac{\pi}{5}\right)  & -\sin\left(\frac{2\pi}{5}\right)  \\
\end{pmatrix}\begin{pmatrix}
\mathcal{J}_{1,5}\\
\mathcal{J}_{2,5}
\end{pmatrix}(\tau).
\end{equation*}
Also we can calculate \cite[Proposition 4.2]{Zw02},
\begin{multline*}
F(\tau) := R_{\frac{1}{4},0}(4\tau) = -i \int_{-4\overline{\tau}}^{i\infty} \frac{g_{\frac{1}{4},0}(w)}{\sqrt{-i(w+4\tau)}} dw = -2i \int_{-\overline{\tau}}^{i\infty}  \frac{g_{\frac{1}{4},0}(4w)}{\sqrt{-i(w+\tau)}} dw \\
= -\frac{i}{2} \int_{-\overline{\tau}}^{i\infty}  \frac{\eta(w)^3}{\sqrt{-i(w+\tau)}} dw,
\end{multline*}
where from the definition we know
\begin{equation*}
g_{\frac{1}{4},0}(4w) = \frac{1}{4} \eta(w)^3.
\end{equation*}
By a direct calculation we obtain
\begin{equation*}
F\left(\tau\right) = - \frac{i}{2} \int_{0}^{i\infty} \frac{\eta(w)^3}{\sqrt{-i(w+\tau)}} dw + \frac{1}{\sqrt{-i\tau} } F\left(-\frac{1}{\tau}\right) .   
\end{equation*}
From Proposition \ref{proposition:RHSshortidentitiesasholpart} and Remark \ref{remark:LHStransfproperties} we have the desired result.
\end{proof}


\section{Identities between Hecke-type double sums} \label{section:proveidentities}

Recall that string functions can be written in terms of Hecke-type sums,
\begin{align*}
\eta(\tau)^3 C_{0,0}^{1/2}(\tau) &= q^{\frac{1}{10}} \left(f_{1,5,20}(q,-q^{12};q)-f_{1,5,20}(1,-q^{8};q) \right),\\
\eta(\tau)^3 C_{0,2}^{1/2}(\tau) &= q^{\frac{1}{10}} \left(f_{1,5,20}(q^2,-q^{16};q)-f_{1,5,20}(q^{-1},-q^{4};q) \right).
\end{align*}

The goal of this section is to prove the the following identities.

\begin{proposition} \label{proposition:hecketypesumsidentities} We have
\begin{align}
\begin{split} \label{equation:strfunc2500short2}
\eta(\tau)^3 C_{0,0}^{1/2}(\tau) &= H_{0,0}^{\operatorname{hol}}(\tau),
\end{split}\\
\begin{split}\label{equation:strfunc2502short2}
\eta(\tau)^3 C_{0,2}^{1/2}(\tau) &= H_{0,2}^{\operatorname{hol}}(\tau),
\end{split}
\end{align}
where
\begin{align}
\begin{split}\label{equation:id1RHSmockForm}
H_{0,0}^{\operatorname{hol}}(\tau) &= q^{\frac{1}{10}} \cdot \left( f_{5,5,1}(q^4,q;q)-qf_{5,5,1}(q^6,q^3;q) \right),
\end{split}\\
\begin{split}\label{equation:id2RHSmockForm}
H_{0,2}^{\operatorname{hol}}(\tau) &= q^{\frac{1}{10}} \cdot \left( q^{-1}f_{5,5,1}(q^3,1;q)-qf_{5,5,1}(q^7,q^4;q) \right).
\end{split}
\end{align}    
\end{proposition}

At first, for the proof of Proposition \ref{proposition:hecketypesumsidentities} we need to present the right-hand side of identities \eqref{equation:strfunc2500short2} and \eqref{equation:strfunc2502short2} as the holomorphic part of indefinite $\vartheta$-functions. Let us define
\begin{align*}
H_{0,0}(\tau) &:= \frac{1}{2} \left(e^{-\frac{3\pi i}{10}}\vartheta_{\begin{psmallmatrix}
    1 \slash 20\\
    1 \slash 4
\end{psmallmatrix},\begin{psmallmatrix}
    1 \slash 10\\
    0
\end{psmallmatrix}}(\tau)-e^{-\frac{7\pi i}{10}}\vartheta_{\begin{psmallmatrix}
    9 \slash 20\\
    1 \slash 4
\end{psmallmatrix},\begin{psmallmatrix}
    1 \slash 10\\
    0
\end{psmallmatrix}}(\tau) \right),\\
H_{0,2}(\tau) &:= \frac{1}{2} \left(e^{-\frac{\pi i}{10}}\vartheta_{\begin{psmallmatrix}
    -3 \slash 20\\
    1 \slash 4
\end{psmallmatrix},\begin{psmallmatrix}
    1 \slash 10\\
    0
\end{psmallmatrix}}(\tau)-e^{-\frac{9\pi i}{10}}\vartheta_{\begin{psmallmatrix}
    13 \slash 20\\
    1 \slash 4
\end{psmallmatrix},\begin{psmallmatrix}
    1 \slash 10\\
    0
\end{psmallmatrix}}(\tau) \right),
\end{align*}
where the quadratic form of the indefinite $\vartheta$-functions is
\begin{equation*}
Q_{R}(x) := \frac{5}{2}x_1^2 + 5x_1x_2 + \frac{1}{2}x_2^2 
\  \ \ \text{for} \ \ \ x=\begin{pmatrix}
 x_1\\
 x_2
\end{pmatrix} \in \R^2
\end{equation*}
and the parameters are
\begin{equation*}
c_1^{R} = \begin{pmatrix}
-1\\
5
\end{pmatrix}, \ \ \ c_2^{R} = \begin{pmatrix}
-1\\
1
\end{pmatrix}.
\end{equation*} 
\begin{proposition} \label{proposition:LHSshortidentitiesasholpart} We have
\begin{equation*}
\begin{pmatrix}
H_{0,0}\\
H_{0,2}
\end{pmatrix}(\tau)
= 
\begin{pmatrix}
H^{\operatorname{hol}}_{0,0}\\
H^{\operatorname{hol}}_{0,2}
\end{pmatrix}(\tau) + R_{\frac{1}{4},0}(4\tau) \cdot
\begin{pmatrix}
\mathcal{J}_{1,5} \\
\mathcal{J}_{2,5}
\end{pmatrix}(\tau).
\end{equation*}    
\end{proposition}

\begin{proof}[Proof of Proposition \ref{proposition:LHSshortidentitiesasholpart}]
In this proof we use the notation $c_1^{R}=c_1$ and $c_2^{R}=c_2$. It is easy to check that $Q(c_1)<0$ , $Q(c_2) < 0$ and $B(c_1,c_2)<0$. Also we can calculate
\begin{align*}
B(c_1,\upsilon) &= 20\upsilon_1,\\ 
B(c_2,\upsilon) &= -4\upsilon_2,
\end{align*}
for $\upsilon = \begin{pmatrix}
    \upsilon_1\\
    \upsilon_2
\end{pmatrix} \in \R^2$. So we can calculate by definition
\begin{align*}
2q^{\frac{1}{10}}e^{\frac{3\pi i}{10}} \cdot f_{5,5,1}(q^4,q;q) &= \vartheta_{\begin{psmallmatrix}
    1 \slash 20\\
    1 \slash 4
\end{psmallmatrix},\begin{psmallmatrix}
    1 \slash 10\\
    0
\end{psmallmatrix}}^{\operatorname{hol}}(\tau),\\
2q^{\frac{1}{10}}e^{\frac{7\pi i}{10}} \cdot qf_{5,5,1}(q^6,q^3;q) &= \vartheta_{\begin{psmallmatrix}
    9 \slash 20\\
    1 \slash 4
\end{psmallmatrix},\begin{psmallmatrix}
    1 \slash 10\\
    0
\end{psmallmatrix}}^{\operatorname{hol}}(\tau),\\
2q^{\frac{9}{10}}e^{\frac{\pi i}{10}} \cdot q^{-1}f_{5,5,1}(q^3,1;q) &= \vartheta_{\begin{psmallmatrix}
    -3 \slash 20\\
    1 \slash 4
\end{psmallmatrix},\begin{psmallmatrix}
    1 \slash 10\\
    0
\end{psmallmatrix}}^{\operatorname{hol}}(\tau),\\
2q^{\frac{9}{10}}e^{\frac{9\pi i}{10}} \cdot qf_{5,5,1}(q^7,q^4;q) &= \vartheta_{\begin{psmallmatrix}
    13 \slash 20\\
    1 \slash 4
\end{psmallmatrix},\begin{psmallmatrix}
    1 \slash 10\\
    0
\end{psmallmatrix}}^{\operatorname{hol}}(\tau).
\end{align*}   
Now let us consider the non-holomorphic parts. We can write
\begin{multline} \label{equation:nonholpartRHS1}
\vartheta_{\begin{psmallmatrix}
    1 \slash 20\\
    1 \slash 4
\end{psmallmatrix},\begin{psmallmatrix}
    1 \slash 10\\
    0
\end{psmallmatrix}}^{\operatorname{nhol}}(\tau)\\
=  -\sum_{\upsilon \in \begin{psmallmatrix}
    1 \slash 20\\
    1 \slash 4
\end{psmallmatrix} + \mathbb{Z}^2} \sgn(B(c_1,\upsilon)) \beta\left(-\frac{B(c_1,\upsilon)^2}{Q(c_1)}\im(\tau)\right) e^{2\pi i B\left(\upsilon,\begin{psmallmatrix}
    1 \slash 10\\
    0
\end{psmallmatrix}\right)} q^{Q(\upsilon)} \\
+\sum_{\upsilon \in \begin{psmallmatrix}
    1 \slash 20\\
    1 \slash 4
\end{psmallmatrix}+ \mathbb{Z}^2} \sgn(B(c_2,\upsilon)) \beta\left(-\frac{B(c_2,\upsilon)^2}{Q(c_2)}\im(\tau)\right) e^{2\pi i B\left(\upsilon,\begin{psmallmatrix}
    1 \slash 10\\
    0
\end{psmallmatrix}\right)} q^{Q(\upsilon)}.
\end{multline}
Using \cite[Proposition 4.3]{Zw02} we see for the first summand in \eqref{equation:nonholpartRHS1}
\begin{align*}
P_0 = \Bigg\{\begin{pmatrix}
    1 \slash 20\\
    1 \slash 4
\end{pmatrix},\Bigg\}, \ 
\langle c_1 \rangle^{\bot}_{\mathbb{Z}} = \Bigg\{ \begin{pmatrix}
0\\
\zeta_2
\end{pmatrix} \ \Bigg| \ \zeta_2 \in \mathbb{Z} \Bigg\},
\end{align*}
and we derive
\begin{multline*}
-\sum_{\upsilon \in \begin{psmallmatrix}
    1 \slash 20\\
    1 \slash 4
\end{psmallmatrix} + \mathbb{Z}^2} \sgn(B(c_1,\upsilon)) \beta\left(-\frac{B(c_1,\upsilon)^2}{Q(c_1)}\im(\tau)\right) e^{2\pi i B\left(\upsilon,\begin{psmallmatrix}
    1 \slash 10\\
    0
\end{psmallmatrix}\right)} q^{Q(\upsilon)} 
\\ = R_{-\frac{1}{20},-2}(20\tau) \cdot \sum_{\zeta \in \frac{1}{2}+\mathbb{Z}} q^{ \frac{\zeta^2}{2}}e^{\pi i \zeta} = 0.
\end{multline*}
Using \cite[Proposition 4.3]{Zw02} we see for the second summand in \eqref{equation:nonholpartRHS1}
\begin{align*}
P_0 = \Bigg\{\begin{pmatrix}
    1 \slash 20\\
    1 \slash 4
\end{pmatrix}\Bigg\},\ 
\langle c_2 \rangle^{\bot}_{\mathbb{Z}} = \Bigg\{ \begin{pmatrix}
\zeta_1\\
0
\end{pmatrix} \ \Bigg| \ \zeta_1 \in \mathbb{Z} \Bigg\},
\end{align*}
and we derive
\begin{multline*}
-\sum_{\upsilon \in \begin{psmallmatrix}
    1 \slash 20\\
    1 \slash 4
\end{psmallmatrix} + \mathbb{Z}^2} \sgn(B(c_2,\upsilon)) \beta\left(-\frac{B(c_2,\upsilon)^2}{Q(c_2)}\im(\tau)\right) e^{2\pi i B\left(\upsilon,\begin{psmallmatrix}
    1 \slash 10\\
    0
\end{psmallmatrix}\right)} q^{Q(\upsilon)} 
\\ = R_{\frac{1}{4},0}(4\tau) \cdot \sum_{\zeta \in \frac{3}{10}+\mathbb{Z}} q^{\frac{5\zeta^2}{2}}e^{5\pi i \zeta} =  R_{\frac{1}{4},0}(4\tau)\cdot e^{\frac{3\pi i}{10}}\mathcal{J}_{1,5}(\tau).
\end{multline*}
Finally we have
\begin{equation*}
 \vartheta_{\begin{psmallmatrix}
    1 \slash 20\\
    1 \slash 4
\end{psmallmatrix},\begin{psmallmatrix}
    1 \slash 10\\
    0
\end{psmallmatrix}}^{\operatorname{nhol}}(\tau) =  e^{\frac{3\pi i}{10}} R_{\frac{1}{4},0}(4\tau)\cdot \mathcal{J}_{1,5}(\tau).  
\end{equation*}
Similarly we can find
\begin{align*}
 \vartheta_{\begin{psmallmatrix}
    9 \slash 20\\
    1 \slash 4
\end{psmallmatrix},\begin{psmallmatrix}
    1 \slash 10\\
    0
\end{psmallmatrix}}^{\operatorname{nhol}}(\tau) &= -e^{\frac{7\pi i}{10}} R_{\frac{1}{4},0}(4\tau) \cdot \mathcal{J}_{1,5}(\tau),\\
 \vartheta_{\begin{psmallmatrix}
    -3 \slash 20\\
    1 \slash 4
\end{psmallmatrix},\begin{psmallmatrix}
    1 \slash 10\\
    0
\end{psmallmatrix}}^{\operatorname{nhol}}(\tau) &= e^{\frac{\pi i}{10}} R_{\frac{1}{4},0}(4\tau) \cdot \mathcal{J}_{2,5}(\tau),\\
 \vartheta_{\begin{psmallmatrix}
    13 \slash 20\\
    1 \slash 4
\end{psmallmatrix},\begin{psmallmatrix}
    1 \slash 10\\
    0
\end{psmallmatrix}}^{\operatorname{nhol}}(\tau) &= -e^{\frac{9\pi i}{10}} R_{\frac{1}{4},0}(4\tau) \cdot \mathcal{J}_{2,5}(\tau).
\end{align*}  
After summing up holomorphic and non-holomorphic parts we obtain the desired result.
\end{proof}

At second, for the proof of Proposition \ref{proposition:hecketypesumsidentities} we need to derive the transformation properties of the indefinite $\vartheta$-functions $H_{0,0}(\tau)$ and $H_{0,2}(\tau)$ on $\Gamma(1)$.
\begin{proposition} \label{proposition:RHStransfproperties} We have
\begin{equation*}
\begin{pmatrix}
H_{0,0}\\
H_{0,2}
\end{pmatrix}(\tau+1) 
\\=  \begin{pmatrix}
\zeta_{10} & 0\\
0 & \zeta_{10}^{-1} \\
\end{pmatrix}
\begin{pmatrix}
H_{0,0}\\
H_{0,2}
\end{pmatrix}(\tau)
\end{equation*}
and
\begin{equation*}
\begin{pmatrix}
H_{0,0}\\
H_{0,2}
\end{pmatrix}\left(-\frac{1}{\tau}\right)
\\=  
(-i\tau) \cdot \frac{2}{\sqrt{5}}\begin{pmatrix}
\sin\left(\frac{2\pi}{5}\right) & -\sin\left(\frac{\pi}{5}\right) \\
-\sin\left(\frac{\pi}{5}\right)  & -\sin\left(\frac{2\pi}{5}\right)  \\
\end{pmatrix}\begin{pmatrix}
H_{0,0}\\
H_{0,2}
\end{pmatrix}(\tau).
\end{equation*}    
\end{proposition}

\begin{proof}[Proof of Proposition \ref{proposition:RHStransfproperties}] Using \cite[Corollary 2.9]{Zw02} it is straightforward to check the transformation property $\tau \rightarrow \tau+1$. Now let us prove the transformation property $\tau \rightarrow -\frac{1}{\tau}$. The matrix of the quadratic form $Q_R$ is
\begin{equation*}
A = \begin{pmatrix}
5 & 5 \\
5 & 1 
\end{pmatrix}.
\end{equation*}
So we have
\begin{equation*}
A^{-1}\mathbb{Z}^2 \ (\operatorname{mod}\mathbb{Z}^2)= \Bigg\{\begin{pmatrix}
\frac{a}{5}-\frac{b}{4}\\
\frac{b}{4}
\end{pmatrix}, \ \ 0\leq a \leq 4, \ \ 0\leq b \leq 3 \Bigg\}.    
\end{equation*}
We apply \cite[Corollary 2.9]{Zw02} and arrive at
\begin{equation*}
 \vartheta_{\begin{psmallmatrix}
1 \slash 10 + a \slash 5-b \slash 4\\
b \slash 4
\end{psmallmatrix}, \begin{psmallmatrix}
    1 \slash 10\\
    0
\end{psmallmatrix}}\left(-\frac{1}{\tau}\right) = \frac{\tau}{\sqrt{20}}\sum_{\substack{0\leq a' \leq 4\\0\leq b' \leq 3}} M_{a,b,a',b'} \cdot \vartheta_{\begin{psmallmatrix}
1 \slash 10 + a' \slash 5-b' \slash 4\\
b' \slash 4
\end{psmallmatrix}, \begin{psmallmatrix}
    1 \slash 10\\
    0
\end{psmallmatrix}}(\tau),
\end{equation*}
where
\begin{equation*}
M_{a,b,a',b'} := \exp \left(2\pi i \left[B\left(\begin{psmallmatrix}
1 \slash 10 + a \slash 5-b \slash 4\\
b \slash 4
\end{psmallmatrix}, \begin{psmallmatrix}
    1 \slash 10\\
    0
\end{psmallmatrix}\right)-B\left(\begin{psmallmatrix}
1 \slash 10 + a' \slash 5-b' \slash 4\\
b' \slash 4
\end{psmallmatrix}, \begin{psmallmatrix}
1 \slash 5 + a \slash 5-b \slash 4\\
b \slash 4
\end{psmallmatrix}\right) \right] \right).
\end{equation*}
So
\begin{align*}
H_{0,0}\left(-\frac{1}{\tau}\right) &= \frac{1}{2} \left(e^{-\frac{3\pi i}{10}}\vartheta_{\begin{psmallmatrix}
    1 \slash 20\\
    1 \slash 4
\end{psmallmatrix},\begin{psmallmatrix}
    1 \slash 10\\
    0
\end{psmallmatrix}}\left(-\frac{1}{\tau}\right)-e^{-\frac{7\pi i}{10}}\vartheta_{\begin{psmallmatrix}
    9 \slash 20\\
    1 \slash 4
\end{psmallmatrix},\begin{psmallmatrix}
    1 \slash 10\\
    0
\end{psmallmatrix}}\left(-\frac{1}{\tau}\right) \right)  \\
&= \frac{\tau}{\sqrt{5}} \sum_{\substack{0\leq a' \leq 4\\0\leq b' \leq 3}} \Tilde{M}_{a',b'} \cdot \vartheta_{\begin{psmallmatrix}
1 \slash 10 + a' \slash 5-b' \slash 4\\
b' \slash 4
\end{psmallmatrix}, \begin{psmallmatrix}
    1 \slash 10\\
    0
\end{psmallmatrix}}(\tau),
\end{align*}
where
\begin{equation*}
\Tilde{M}_{a',b'} = \frac{1}{4} \left( e^{-\frac{3\pi i}{10}} \cdot M_{1,1,a',b'} - e^{-\frac{7\pi i}{10}}  \cdot  M_{3,1,a',b'} \right).
\end{equation*}
Using \cite[Corollary 2.9]{Zw02} we can write
\begin{align*}
\sum_{\substack{0\leq a' \leq 4\\0\leq b' \leq 3}} \Tilde{M}_{a',b'} &\cdot \vartheta_{\begin{psmallmatrix}
1 \slash 10 + a' \slash 5-b' \slash 4\\
b' \slash 4
\end{psmallmatrix}, \begin{psmallmatrix}
    1 \slash 10\\
    0
\end{psmallmatrix}}(\tau) = \\
&\left(\Tilde{M}_{0,0} - e^{-\frac{\pi i}{5}}\Tilde{M}_{4,0}\right) \cdot \vartheta_{\begin{psmallmatrix}
1 \slash 10\\
0
\end{psmallmatrix}, \begin{psmallmatrix}
    1 \slash 10\\
    0
\end{psmallmatrix}}(\tau)+\left(\Tilde{M}_{1,0} - e^{-\frac{3\pi i}{5}}\Tilde{M}_{3,0}\right) \cdot \vartheta_{\begin{psmallmatrix}
3 \slash 10\\
0
\end{psmallmatrix}, \begin{psmallmatrix}
    1 \slash 10\\
    0
\end{psmallmatrix}}(\tau)\\
&+\Tilde{M}_{2,0} \cdot \vartheta_{\begin{psmallmatrix}
1 \slash 2\\
0
\end{psmallmatrix}, \begin{psmallmatrix}
    1 \slash 10\\
    0
\end{psmallmatrix}}(\tau)+(\Tilde{M}_{0,1}- e^{-\frac{\pi i}{5}}\Tilde{M}_{4,3})\cdot \vartheta_{\begin{psmallmatrix}
-3 \slash 20\\
1 \slash 4\\
\end{psmallmatrix}, \begin{psmallmatrix}
    1 \slash 10\\
    0
\end{psmallmatrix}}(\tau)\\
&+(\Tilde{M}_{1,1}-e^{-\frac{3\pi i}{5}} \Tilde{M}_{3,3})\cdot \vartheta_{\begin{psmallmatrix}
1 \slash 20\\
1 \slash 4\\
\end{psmallmatrix}, \begin{psmallmatrix}
    1 \slash 10\\
    0
\end{psmallmatrix}}(\tau)+(\Tilde{M}_{2,1}+ \Tilde{M}_{2,3})\cdot \vartheta_{\begin{psmallmatrix}
1 \slash 4\\
1 \slash 4\\
\end{psmallmatrix}, \begin{psmallmatrix}
    1 \slash 10\\
    0
\end{psmallmatrix}}(\tau)\\
&+(\Tilde{M}_{3,1}- e^{-\frac{7\pi i}{5}}\Tilde{M}_{1,3})\cdot \vartheta_{\begin{psmallmatrix}
9 \slash 20\\
1 \slash 4\\
\end{psmallmatrix}, \begin{psmallmatrix}
    1 \slash 10\\
    0
\end{psmallmatrix}}(\tau)+(\Tilde{M}_{4,1}-e^{-\frac{9\pi i}{5}} \Tilde{M}_{0,3})\cdot \vartheta_{\begin{psmallmatrix}
13 \slash 20\\
1 \slash 4\\
\end{psmallmatrix}, \begin{psmallmatrix}
    1 \slash 10\\
    0
\end{psmallmatrix}}(\tau)\\
&+ \left(\Tilde{M}_{0,2} - e^{-\frac{\pi i}{5}}\Tilde{M}_{4,2}\right) \cdot \vartheta_{\begin{psmallmatrix}
-1 \slash 5\\
1 \slash 2
\end{psmallmatrix}, \begin{psmallmatrix}
    1 \slash 10\\
    0
\end{psmallmatrix}}(\tau)+\left(\Tilde{M}_{1,2} - e^{-\frac{3\pi i}{5}}\Tilde{M}_{3,2}\right) \cdot \vartheta_{\begin{psmallmatrix}
-2 \slash 5\\
1 \slash 2
\end{psmallmatrix}, \begin{psmallmatrix}
    1 \slash 10\\
    0
\end{psmallmatrix}}(\tau)\\
&+\Tilde{M}_{2,2} \cdot \vartheta_{\begin{psmallmatrix}
0\\
1\slash 2
\end{psmallmatrix}, \begin{psmallmatrix}
    1 \slash 10\\
    0
\end{psmallmatrix}}(\tau).
\end{align*}
The coefficients vanish and by collecting terms we have
\begin{equation*}
\sum_{\substack{0\leq a' \leq 4\\0\leq b' \leq 3}} \Tilde{M}_{a',b'} \cdot \vartheta_{\begin{psmallmatrix}
1 \slash 10 + a' \slash 5-b' \slash 4\\
b' \slash 4
\end{psmallmatrix}, \begin{psmallmatrix}
    1 \slash 10\\
    0
\end{psmallmatrix}}(\tau) = -i \sin\left(\frac{2\pi}{5}\right)\cdot  H_{0,0}(\tau) +  i \sin\left(\frac{\pi}{5}\right)\cdot H_{0,2}(\tau).   
\end{equation*}
So we obtain the transformation property
\begin{equation*}
H_{0,0}\left(-\frac{1}{\tau}\right) = (-i\tau) \cdot \frac{2}{\sqrt{5}} \left( \sin\left(\frac{\pi}{5}\right)\cdot H_{0,0}(\tau)-\sin\left(\frac{2\pi}{5}\right)\cdot H_{0,2}(\tau)\right).  
\end{equation*}
Similarly we can obtain
\begin{equation*}
H_{0,2}\left(-\frac{1}{\tau}\right) = (-i\tau) \cdot \frac{2}{\sqrt{5}} \left( -\sin\left(\frac{2\pi}{5}\right)\cdot H_{0,0}(\tau)-\sin\left(\frac{\pi}{5}\right)\cdot H_{0,2}(\tau)\right).  \qedhere
\end{equation*}
\end{proof}

Using the previous results we are able to prove \eqref{equation:strfunc2500short2} and \eqref{equation:strfunc2502short2}. 

\begin{proof}[Proof of Proposition \ref{proposition:hecketypesumsidentities}]
Let us denote
\begin{equation*}
G(\tau) = \begin{pmatrix}
G_1\\
G_2 
\end{pmatrix}(\tau) = \eta(\tau)^3 \cdot \begin{pmatrix}
  C^{1/2}_{0,0}\\
  C^{1/2}_{0,2}
\end{pmatrix}(\tau) -  \begin{pmatrix}
  H^{\operatorname{hol}}_{0,0}\\
  H^{\operatorname{hol}}_{0,2}
\end{pmatrix}(\tau) .    
\end{equation*}
By Proposition \ref{proposition:RHSshortidentitiesasholpart} and Proposition \ref{proposition:LHSshortidentitiesasholpart} the non-holomorphic parts of the summands cancel each other and we have 
\begin{equation*}
G(\tau) = \begin{pmatrix}
G_1\\
G_2 
\end{pmatrix}(\tau) = \begin{pmatrix}
\vartheta_{\begin{psmallmatrix}
    0\\
    1 \slash 10
\end{psmallmatrix},\begin{psmallmatrix}
    0\\
    1 \slash 2
\end{psmallmatrix}}\\
\vartheta_{\begin{psmallmatrix}
    0\\
    3 \slash 10
\end{psmallmatrix},\begin{psmallmatrix}
    0\\
    1 \slash 2
\end{psmallmatrix}}
\end{pmatrix}(\tau) -  \begin{pmatrix}
  H_{0,0}\\
  H_{0,2}
\end{pmatrix}(\tau) .    
\end{equation*}
Hence by Remark \ref{remark:LHStransfproperties} and Proposition \ref{proposition:RHStransfproperties} we have the modular transformation property
\begin{equation*}
\begin{pmatrix}
G_{0,0}\\
G_{0,2}
\end{pmatrix}\left(-\frac{1}{\tau}\right)
\\=  
(-i\tau) \cdot \frac{2}{\sqrt{5}}\begin{pmatrix}
\sin\left(\frac{2\pi}{5}\right) & -\sin\left(\frac{\pi}{5}\right) \\
-\sin\left(\frac{\pi}{5}\right)  & -\sin\left(\frac{2\pi}{5}\right)  \\
\end{pmatrix}\begin{pmatrix}
G_{0,0}\\
G_{0,2}
\end{pmatrix}(\tau). 
\end{equation*}
We have that $G$ transforms on $\Gamma(1)$ with the same representation as
\begin{equation*}
\eta(\tau) \cdot \begin{pmatrix}
\mathcal{J}_{5,1}\\
\mathcal{J}_{5,2}
\end{pmatrix}(\tau),   
\end{equation*}
Using \cite[Lemma 2.1]{Biag} we derive that $G_1$ and $G_2$ are holomorphic modular forms of weight $1$ on $\Gamma := \Gamma_1(5)$ with some multiplier system. Let us show that $G(\tau)$ is holomorphic at the cusps of $\Gamma$. In the notation of \cite{FG} we can calculate
\begin{equation*}
\operatorname{ORD}(G_i,\infty,\Gamma) > 0. 
\end{equation*}
Using the transformation properties of $G$ on $\Gamma(1)$ we know for any cusp $\gamma(\infty) \ (\operatorname{mod}\Gamma)$ that
\begin{equation*}
\operatorname{ORD}(G_i,\gamma(\infty),\Gamma) = \operatorname{ORD}(G_i|_{1}\gamma,\infty,\Gamma) > \min(\operatorname{ORD}(G_1,\infty,\Gamma),\operatorname{ORD}(G_2,\infty,\Gamma)) > 0,
\end{equation*}
where we use the slash operator
\begin{equation*}
     (f |_{k} \gamma)(\tau) =   (ad-bc)^{\frac{k}{2}}(c\tau+d)^{-k}f(\gamma\tau) \ \ \text{for} \ \ \gamma = \begin{pmatrix}
        a & b\\
        c & d
    \end{pmatrix}. 
\end{equation*}
By calculating
\begin{equation*}
[\widehat{\Gamma(1)}:\widehat{\Gamma_1(5)}] =  [\widehat{\Gamma(1)}:\Gamma_1(5)] = \frac{1}{2}[\Gamma(1):\Gamma_1(5)] = 12
\end{equation*}
and applying the Valence formula \cite[Theorem 2.4]{FG} we just need to verify that $\operatorname{ORD}(G_i,\infty,\Gamma) > 1$ in a computing environment and we obtain $G_1=G_2=0$.
\end{proof}


\section{The mock theta function expression for $1/2$-level string functions} \label{section:fractionLevelPositive12}

We start by giving a sketch the proof, where in the sketch we focus on the case $\ell=0$.  The general string functions are defined in terms of $f_{1,5,20}(x,y;q)$'s, for example
\begin{equation*}
(q)_{\infty}^3 \mathcal{C}_{2k,0}^{1/2}(q)= 
F_{0}(k):=f_{1,5,20}(q^{1+k},-q^{12};q) -f_{1,5,20}(q^k,-q^{8};q),
\end{equation*}
which as pointed out in the introduction is difficult to evaluate in terms of a workable expression consisting of Appell functions and theta functions.  This leads us to find double-sum expressions of a more symmetric form, which are easier to evaluate.  For example, we will show
\begin{equation*}
 (q)_{\infty}^3 \mathcal{C}_{2k,0}^{1/2}(q) 
=q^{k}f_{5,5,1}(q,q^{4k-2};q)
+q^{3k}f_{5,5,1}(q^6,q^{4k+1};q)(k)=:G_{0}(k),
\end{equation*}
which one can express in terms of Appell functions and theta functions, see Lemma \ref{lemma:f551-expansion}, and for the general form see \cite[Theorem $1.4$]{HM}.  Showing that $F_{0}(k)=G_{0}(k)$ follows from demonstrating that both functions $F_{0}(k)$ and $G_{0}(k)$ satisfy the same recurrence relation and the same initial condition.  In particular, we will demonstrate that
\begin{equation*}
F_{0}(k)-G_{0}(k)
=q^{-4(k+1)+2}\left (F_{0}(k+1)-G_{0}(k+1) \right ) 
\end{equation*}
and that
\begin{equation*}
F_{0}(0)=G_{0}(0).
\end{equation*}

We proceed to the general proof.  For even spin $\ell$, we define the following:
\begin{align}
 (q)_{\infty}^3 \mathcal{C}_{2k,0}^{1/2}(q) 
=
F_{0}(k)
&:=f_{1,5,20}(q^{1+k},-q^{12};q) -f_{1,5,20}(q^k,-q^{8};q),\\
G_{0}(k)
&:=q^{k}f_{5,5,1}(q,q^{4k-2};q)
+q^{3k}f_{5,5,1}(q^6,q^{4k+1};q),\notag\\
 (q)_{\infty}^3 \mathcal{C}_{2k,2}^{1/2}(q)=
F_{2}(k)
& :=f_{1,5,20}(q^{2+k},-q^{16};q)-f_{1,5,20}(q^{k-1},-q^{4};q),\\
G_{2}(k)
&:=q^{k-1}f_{5,5,1}(q^3,q^{4k};q)
+q^{3k+1}f_{5,5,1}(q^{8},q^{4k+3};q).\notag
\end{align}
For odd spin $\ell$, we define
\begin{align}
 (q)_{\infty}^3 \mathcal{C}_{2k+1,1}^{1/2}(q) =
 F_{1}(k)
& :=f_{1,5,20}(q^{2+k},-q^{14};q) - f_{1,5,20}(q^k,-q^{6};q),\\
G_{1}(k)&:=q^{k}f_{5,5,1}(q^3,q^{4k+2};q)
+q^{3k+3}f_{5,5,1}(q^8,q^{4k+5};q),\notag\\
 (q)_{\infty}^3 \mathcal{C}_{2k+1,3}^{1/2}(q) =
 F_{3}(k)
& :=f_{1,5,20}(q^{3+k},-q^{18};q) - f_{1,5,20}(q^{k-1},-q^{2};q),\\
G_{3}(k)&:=q^{k-1}f_{5,5,1}(q,q^{4k};q)
 +q^{3k}f_{5,5,1}(q^6,q^{4k+3};q).\notag
\end{align}

We have two types of recurrence relations for our four families of string functions parametrized by $\ell\in\{0,1,2,3\}$.  They read
\begin{proposition} \label{proposition:recRelationSpin02}
For $\ell\in\{0,2\}$ we have the following recurrence relation:
\begin{equation}
F_{\ell}(k)-G_{\ell}(k)=q^{-4(k+1)+2}
\left (F_{\ell}(k+1) - G_{\ell}(k+1)\right ). 
\label{equation:recRelationSpin02}
\end{equation}
\end{proposition}

\begin{proposition} \label{proposition:recRelationSpin13}
For $\ell\in\{1,3\}$ we have the following recurrence relation:
\begin{equation}
F_{\ell}(k)-G_{\ell}(k)=q^{-4(k+1)}
\left ( F_{\ell}(k+1)-G_{\ell}(k+1)\right ).
\label{equation:recRelationSpin13}
\end{equation}
\end{proposition}
We will prove Propositions \ref{proposition:recRelationSpin02} and \ref{proposition:recRelationSpin13} in Section \ref{section:recRelations}.  

Note that the recurrence relations can be written in terms of the $k=0$ evaluations.

\begin{remark} For $\ell\in\{0,2\}$ we have that
\begin{equation*}
F_{\ell}(k)-G_{\ell}(k)
=q^{2k^2}\left ( F_{\ell}(0)-G_{\ell}(0)\right).
\end{equation*}
For $\ell\in\{1,3\}$ we have that
\begin{equation*}
F_{\ell}(k)-G_{\ell}(k)
=q^{2k^2+2k}\left ( F_{\ell}(0)-G_{\ell}(0)\right).
\end{equation*}
\end{remark}

We see that both types of expressions satisfy the same initial conditions:
\begin{proposition}\label{proposition:FGidk0} For $\ell\in\{ 0,1,2,3\}$ we have
\begin{equation}
F_{\ell}(0)=G_{\ell}(0).\label{equation:FGidk0}
\end{equation}
\end{proposition}
We will prove Proposition \ref{proposition:FGidk0} in Section \ref{section:initValues}.  As an immediate corollary, we have that both expressions satisfy the same recurrence relations and the same initial conditions, and are thus equal:
\begin{proposition} For $\ell\in\{ 0,1,2,3\}$ we have
\begin{equation}
F_{\ell}(k)=G_{\ell}(k).
\end{equation}
\end{proposition}

Now it is possible to compute nice compact expressions for the general expressions of the four families of string functions in terms of Appell functions and theta functions:
\begin{proposition} \label{proposition:G14-fullAppellTheta} We have the following
\begin{align}
 (q)_{\infty}^3 \mathcal{C}_{2k,0}^{1/2}(q) 
 &=q^{k}j(q;q^5)\left ( m(-q^{1+4k},-1;q^4)-q^{2k-1}m(-q^{-1+4k},-1;q^4) \right )\label{equation:G0-fullAppellTheta}\\
 &\qquad -q^{2k^2-1}\frac{J_{20}^3}{\overline{J}_{0,4}\overline{J}_{0,20}}
\sum_{d=0}^{4}q^{2d(d-1)}
\frac{j\big (q^{2+4d};q^{5}\big )  j\big (-q^{1+4d};q^{20}\big )j\big (q^{2+4d};q^{20}\big )}
{J_{1,20}j(q^{1+4d};q^{20}\big )}\notag \\
&\qquad +q^{2k^2}\frac{J_{20}^3}{\overline{J}_{0,4}\overline{J}_{0,20}}
\sum_{d=0}^{4}q^{2d(d+1)}
 \frac{j\big (q^{4+4d};q^{5}\big )  j\big (-q^{3+4d};q^{20}\big )j\big (q^{14+4d};q^{20}\big )}
{J_{11,20}j(q^{3+4d};q^{20}\big )},\notag
\end{align}

\begin{align}
 (q)_{\infty}^3 \mathcal{C}_{2k,2}^{1/2}(q)
 &=q^{k-1}j(q^2;q^5)\left ( m(-q^{1+4k},-1;q^4)-q^{2k-1}m(-q^{-1+4k},-1;q^4)\right )\label{equation:G2-fullAppellTheta}\\
 &\qquad -q^{2k^2-1}\frac{J_{20}^3}{\overline{J}_{0,4}\overline{J}_{0,20}}\sum_{d=0}^{4}q^{2d(d+1)}
\frac{j\big (q^{4+4d};q^{5}\big )  j\big (-q^{1+4d};q^{20}\big )j\big (q^{14+4d};q^{20}\big )}
{J_{7,20}j(q^{1+4d};q^{20}\big )}\notag\\
&\qquad -q^{2k^2-2}\frac{J_{20}^3}{\overline{J}_{0,4}\overline{J}_{0,20}}
\sum_{d=0}^{4}
q^{2d(d-1)} \frac{j\big (q^{1+4d};q^{5}\big )  j\big (-q^{3+4d};q^{20}\big )j\big (q^{6+4d};q^{20}\big )}
{J_{3,20}j(q^{3+4d};q^{20}\big )},\notag
\end{align}

\begin{align}
 (q)_{\infty}^3 &\mathcal{C}_{2k+1,1}^{1/2}(q) 
 \label{equation:G1-fullAppellTheta}\\
 &=q^kj(q^2;q^5)\left ( m(-q^{1+2(2k+1)},-1;q^4)-q^{(2k+1)-1}m(-q^{-1+2(2k+1)},-1;q^4)\right ) \notag\\
 &\qquad +q^{2k^2+2k-1}\frac{J_{20}^3}{\overline{J}_{0,4}\overline{J}_{0,20}}
\sum_{d=0}^{4}q^{2d(d-1)}
\frac{j\big (q^{1+4d};q^{5}\big )  j\big (-q^{3+4d};q^{20}\big )j\big (q^{6+4d};q^{20}\big )}
{J_{3,20}j(q^{3+4d};q^{20}\big )}\notag\\
&\qquad +q^{2k^2+2k}\frac{J_{20}^3}{\overline{J}_{0,4}\overline{J}_{0,20}}
\sum_{d=0}^{4}
q^{2d(d+1)}
 \frac{j\big (q^{4+4d};q^{5}\big ) j\big (-q^{1+4d};q^{20}\big )j\big (q^{14+4d};q^{20}\big )}
{J_{7,20}j(q^{1+4d};q^{20}\big )},\notag
\end{align}

\begin{align}
 (q)_{\infty}^3& \mathcal{C}_{2k+1,3}^{1/2}(q) 
 \label{equation:G3-fullAppellTheta}\\
 &=q^{k-1}j(q;q^5)\left ( m(-q^{1+2(2k+1)},-1;q^4)-q^{(2k+1)-1}m(-q^{-1+2(2k+1)},-1;q^4)\right )\notag \\
 &\qquad -q^{2k^2+2k-1}\frac{J_{20}^3}{\overline{J}_{0,4}\overline{J}_{0,20}}
\sum_{d=0}^{4}q^{2d(d+1)}
 \frac{j\big (q^{4+4d};q^{5}\big )  j\big (-q^{3+4d};q^{20}\big )j\big (q^{14+4d};q^{20}\big )}
{J_{11,20}j\big (q^{3+4d};q^{20}\big )}\notag\\
&\qquad +q^{2k^2+2k-2}\frac{J_{20}^3}{\overline{J}_{0,4}\overline{J}_{0,20}}
\sum_{d=0}^{4}q^{2d(d-1)}
 \frac{j\big (q^{2+4d};q^{5}\big )  j\big (-q^{1+4d};q^{20}\big )j\big (q^{2+4d};q^{20}\big )}
{J_{1,20}j\big (q^{1+4d};q^{20}\big)}.\notag
\end{align}
\end{proposition}
We will prove Proposition \ref{proposition:G14-fullAppellTheta} in Section \ref{section:f551Evaluations}.  The proof of Theorem \ref{theorem:fractionalLevelPlus12} is now straightforward.

\begin{proof}[Proof of Theorem \ref{theorem:fractionalLevelPlus12}]
The first step is to use Proposition \ref{proposition:genAppellToMockMu} in order to write the Appell function expressions in (\ref{equation:G0-fullAppellTheta}), (\ref{equation:G2-fullAppellTheta}), (\ref{equation:G1-fullAppellTheta}), and (\ref{equation:G3-fullAppellTheta}) in terms of $\mu(q)$.  We then rewrite the resulting collection of theta functions as as single quotient using Lemma {\ref{lemma:generalSingleQuotienEll03}} for $\ell\in\{0,3\}$ or Lemma {\ref{lemma:generalSingleQuotienEll12}} for $\ell\in\{1,2\}$.\qedhere

\end{proof}

\begin{proof}[Proof of Corollary \ref {corollary:fractionalLevelPlus12-2ndA}]   From \cite[$(3.28)$]{AM}, we have
\begin{equation}
\mu(q)+4A(-q)=\frac{J_{1}^5}{J_{2}^4}.\label{equation:muA-id}
\end{equation}
Substituting \eqref{equation:muA-id} into Theorem \ref{theorem:fractionalLevelPlus12}, we see that it suffices to prove
\begin{align*}
\frac{1}{2}j(q;q^5)\frac{J_{1}^5}{J_{2}^4}
+\frac{1}{2}\cdot \frac{J_{1}^3J_{10}^3}{J_{4}J_{5}}\cdot \frac{1}{J_{1,10}J_{8,20}}
= \frac{J_{1}^4J_{4}J_{8,20}}{J_{2}^4},\\
\frac{1}{2q}j(q^2;q^5)\frac{J_{1}^5}{J_{2}^4}
-\frac{1}{2q}\cdot \frac{J_{1}^3J_{10}^3}{J_{4}J_{5}}\cdot \frac{1}{J_{3,10}J_{4,20}}
= -\frac{J_{1}^4J_{4}J_{4,20}}{J_{2}^4},
\end{align*}
 both of which are easily shown using the methods of \cite{FG} as demonstrated in the proof of Lemma {\ref{lemma:generalSingleQuotienEll03}}.
 \end{proof}


\subsection{The recurrence relations}  \label{section:recRelations}  We prove the four families of recurrence relations found Propositions \ref{proposition:recRelationSpin02} and \ref{proposition:recRelationSpin13}.  They will immediately follow from the following lemma:
\begin{lemma}\label{lemma:recRelationSpin0213} There holds
\begin{align}
F_0(k)&=q^{-4(k+1)+2}F_0(k+1)
+q^{-k}j(q;q^5)-q^{-3k-1}j(q;q^5),
\label{equation:F0-functionalEquation}\\
G_0(k)&=q^{-4(k+1)+2}G_0(k+1) 
-q^{-3k-1}j(q;q^5) +q^{-k}j(q;q^5),
\label{equation:G0-functionalEquation}\\
F_2(k)&=q^{-4(k+1)+2}F_2(k+1)
+q^{-1-k}j(q^2;q^5)-q^{-3k-2}j(q^2;q^5),
\label{equation:F2-functionalEquation}\\
G_2(k)&=q^{-4(k+1)+2}G_2(k+1)
-q^{-3k-2}j(q^{2};q^{5}) +q^{-k-1}j(q^{2};q^{5}),
\label{equation:G2-functionalEquation}\\
F_{1}(k)&=q^{-4(k+1)}F_{1}(k+1)
+q^{-1-k}j(q^{2};q^5) -q^{-3k-3}j(q^{2};q^5),
\label{equation:F1-functionalEquation}\\
G_{1}(k)&=q^{-4(k+1)}G_{1}(k+1)
-q^{-3k-3}j(q^{2};q^{5})  +q^{-k-1}j(q^{2};q^5),
\label{equation:G1-functionalEquation}\\
F_{3}(k)&=q^{-4(k+1)}F_{3}(k+1)
+ q^{-2-k}j(q;q^5) -q^{-3k-4}j(q;q^{5}),
\label{equation:F3-functionalEquation}\\
G_{3}(k)&=q^{-4(k+1)}G_{3}(k+1) 
-q^{-3k-4}j(q;q^{5}) +q^{-k-2}j(q;q^{5})
\label{equation:G3-functionalEquation}.
\end{align}
\end{lemma}

\begin{proof} [Proof of Lemma \ref{lemma:recRelationSpin0213}]

The proofs are all very similar, so we will only demonstrate the first two identities (\ref{equation:F0-functionalEquation}) and (\ref{equation:G0-functionalEquation}).  We recall that
\begin{gather*}
F_{0}(k):=f_{1,5,20}(q^{1+k},-q^{12};q) -f_{1,5,20}(q^k,-q^{8};q),\\
G_{0}(k):=q^{k}f_{5,5,1}(q,q^{4k-2};q)+q^{3k}f_{5,5,1}(q^6,q^{4k+1};q),
\end{gather*}
and specialize Proposition \ref{proposition:f-functionaleqn} with $(\ell,k)=(1,-1)$ to get
{\allowdisplaybreaks \begin{align*}
G_0(k)&=q^{k}\Big [ (-q)(-q^{4k-2})^{-1}q^{-4}f_{5,5,1}(q,q^{4}q^{4k-2};q)
+j(q^{4k-2};q)\\
&\qquad \qquad +\sum_{m=0}^{-2}(-q^{4k-2})^mq^{\binom{m}{2}}j(q^{5m+1};q^5)\Big ] \\
&\qquad + q^{3k}\Big [ (-q^6)(-q^{4k+1})^{-1}q^{-4}f_{5,5,1}(q^6,q^{4}q^{4k+1};q)
+j(q^{4k+1};q)\\
&\qquad \qquad +\sum_{m=0}^{-2}(-q^{4k+1})^mq^{\binom{m}{2}}j(q^{5m+6};q^5)\Big ].
\end{align*}}%
Simplifying and recalling the summation convention (\ref{equation:sumconvention}) yields
\begin{align*}
G_0(k)&=q^{-3(k+1)+2}f_{5,5,1}(q,q^{4(k+1)-2};q) -q^k\sum_{m=-1}^{-1}(-q^{4k-2})^mq^{\binom{m}{2}}j(q^{5m+1};q^5) \\
&\qquad + q^{-(k+1)+2}f_{5,5,1}(q^6,q^{4(k+1)+1};q)
-q^{3k}\sum_{m=-1}^{-1}(-q^{4k+1})^mq^{\binom{m}{2}}j(q^{5m+6};q^5),
\end{align*}
or equivalently,
\begin{equation*}
G_0(k)=q^{-4(k+1)+2}G_0(k+1) -q^{-3k-1}j(q;q^5) +q^{-k}j(q;q^5),
\end{equation*}
which is (\ref{equation:G0-functionalEquation}).  Now, specializing Proposition \ref{proposition:f-functionaleqn} with $(\ell,k)=(-4,1)$ results in
{\allowdisplaybreaks \begin{align*}
F_0(k)&=(-q^{1+k})^{-4}(q^{12})q^{\binom{5}{2}-20}f_{1,5,20}(q^{1+(k+1)},-q^{12};q)\\
&\qquad +\sum_{m=0}^{-5}(-q^{1+k})^{m}q^{\binom{m}{2}}j(-q^{5m+12};q^{20})+j(q^{1+k};q)\\
&\qquad - (-q^{k})^{-4}(q^{8})q^{\binom{5}{2}-20}f_{1,5,20}(q^{(k+1)},-q^{8};q)\\
&\qquad -\sum_{m=0}^{-5}(-q^{k})^{m}q^{\binom{m}{2}}j(-q^{5m+8};q^{20})-j(q^{k};q).
\end{align*}}%
Noting that $j(q^n;q)=0$ for $n\in\mathbb{Z}$ and again recalling the summation convention (\ref{equation:sumconvention}) gives
{\allowdisplaybreaks \begin{align*}
F_0(k)&=q^{-4(k+1)+2}F_0(k+1)\\
&\qquad -\sum_{m=-4}^{-1}(-q^{1+k})^{m}q^{\binom{m}{2}}j(-q^{5m+12};q^{20})
+\sum_{m=-4}^{-1}(-q^{k})^{m}q^{\binom{m}{2}}j(-q^{5m+8};q^{20})\\
&=q^{-4(k+1)+2}F_0(k+1)\\
&\qquad -\Big [ (-q^{1+k})^{-1}q^{\binom{-1}{2}}j(-q^{-5+12};q^{20})
+(-q^{1+k})^{-2}q^{\binom{-2}{2}}j(-q^{-10+12};q^{20})\\
&\qquad \qquad (-q^{1+k})^{-3}q^{\binom{-3}{2}}j(-q^{-15+12};q^{20})
+(-q^{1+k})^{-4}q^{\binom{-4}{2}}j(-q^{-20+12};q^{20})\Big ] \\
&\qquad +\Big [ (-q^{k})^{-1}q^{\binom{-1}{2}}j(-q^{-5+8};q^{20})
+ (-q^{k})^{-2}q^{\binom{-2}{2}}j(-q^{-10+8};q^{20})\\
&\qquad \qquad + (-q^{k})^{-3}q^{\binom{-3}{2}}j(-q^{-15+8};q^{20})
+ (-q^{k})^{-4}q^{\binom{-4}{2}}j(-q^{-20+8};q^{20})
\Big ]. 
\end{align*}}%
Using (\ref{equation:j-elliptic}) and simplifying yields
{\allowdisplaybreaks \begin{align*}
F_0(k)&=q^{-4(k+1)+2}F_0(k+1)\\
&\qquad -\Big [ -q^{-k}j(-q^{7};q^{20})
+q^{-2k+1}j(-q^{2};q^{20})
-q^{-3k}j(-q^{3};q^{20})
+q^{-4k-2}j(-q^{8};q^{20})\Big ] \\
&\qquad +\Big [ -q^{-k+1}j(-q^{3};q^{20})
+ q^{-2k+1}j(-q^{2};q^{20})
 -q^{-3k-1}j(-q^{7};q^{20})
+ q^{-4k-2}j(-q^{12};q^{20})
\Big ]\\
&=q^{-4(k+1)+2}F_0(k+1)\\
&\qquad  +q^{-k}j(-q^{7};q^{20})
+q^{-3k}j(-q^{3};q^{20}) -q^{-k+1}j(-q^{3};q^{20})
 -q^{-3k-1}j(-q^{7};q^{20}).
\end{align*}}%
Using (\ref{equation:jsplit-m2}) brings us to
 \begin{equation*}
F_0(k)=q^{-4(k+1)+2}F_0(k+1)
   +q^{-k}j(q;q^5)-q^{-3k-1}j(q;q^5),
\end{equation*}
which is (\ref{equation:F0-functionalEquation}).
\end{proof}


\subsection{The initial values} \label{section:initValues}

We record the appropriate specialization of Proposition \ref{proposition:H7eq1.14}
\begin{equation}
f_{5,5,1}(x,y;q)=-\frac{q^{11}}{xy}f_{5,5,1}(q^{15}/x,q^{7}/y;q),
\label{equation:f551-fnqEqnFlip}
\end{equation} 
and the appropriate specializations of Corollary \ref{corollary:fabc-funceqnspecial}
\begin{align}
f_{5,5,1}(x,y;q)&=-yf_{5,5,1}(q^5x,qy;q)+j(x;q^5),
\label{equation:f551-fnqEqn1}\\
f_{5,5,1}(x,y;q)&=-xf_{5,5,1}(q^5x,q^5y;q)+j(y;q).
\label{equation:f551-fnqEqn2}
\end{align}

\begin{proof}[Proof of Proposition \ref{proposition:FGidk0}]
We point out that the two identities \eqref{equation:strfunc2500short2} and \eqref{equation:strfunc2502short2} found in Proposition \ref{proposition:hecketypesumsidentities} are required for all four initial conditions.  The proofs are all similar, so we will only prove (\ref{equation:FGidk0}) in the cases $\ell=0$ and $\ell=1$.  By definition and \eqref{equation:strfunc2500short2}, we have
\begin{align*}
F_{0}(0)
&=f_{1,5,20}(q,-q^{12};q) -f_{1,5,20}(q,-q^{8};q)\\
&=f_{5,5,1}(q^4,q;q)-qf_{5,5,1}(q^6,q^3;q).
\end{align*}
Using (\ref{equation:f551-fnqEqn2}) and the fact that $j(q^n;q)=0$ for $n\in\mathbb{Z}$ gives
\begin{align*}
F_{0}(0)
&=-q^4f_{5,5,1}(q^9,q^6;q)+j(q;q)+ f_{5,5,1}(q,q^{-2};q)-j(q^{-2};q)\\
&=-q^4f_{5,5,1}(q^9,q^6;q)+ f_{5,5,1}(q,q^{-2};q).
\end{align*}
Applying (\ref{equation:f551-fnqEqnFlip}) gives the result
\begin{equation*}
F_{0}(0)=f_{5,5,1}(q^6,q;q)+ f_{5,5,1}(q,q^{-2};q)=G_{0}(0).
\end{equation*}


For the case $\ell=1$, we rearrange (\ref{eq:crossspiniden0211}) to write
\begin{equation*} 
(q)_{\infty}^3 \mathcal{C}_{1,1}^{1/2}(q)
=-q(q)_{\infty}^3 \mathcal{C}_{0,2}^{1/2}(q)+j(q^2;q^5).
\end{equation*}
From the definition and (\ref{equation:strfunc2502short2}), we can write
\begin{align*}
F_{1}(0)
&=f_{1,5,20}(q^{2},-q^{14};q) - f_{1,5,20}(1,-q^{6};q)\\
&=-q\left ( q^{-1}f_{5,5,1}(q^3,1;q)+qf_{5,5,1}(q^8,q^3;q)\right ) +j(q^2;q^5)\\
&=-f_{5,5,1}(q^3,1;q)-q^2f_{5,5,1}(q^8,q^3;q) +j(q^2;q^5).
\end{align*}
Using (\ref{equation:f551-fnqEqn2}) and noting that $j(q^n;q)=0$ for $n\in\mathbb{Z}$ yields
\begin{align*}
F_{1}(0)
&=q^3f_{5,5,1}(q^8,q^5;q)-j(1;q)-q^2f_{5,5,1}(q^8,q^3;q) +j(q^2;q^5)\\
&=q^3f_{5,5,1}(q^8,q^5;q)-q^2f_{5,5,1}(q^8,q^3;q) +j(q^3;q^5).
\end{align*}
Using (\ref{equation:f551-fnqEqn1}) gives the result
\begin{equation*}
F_{1}(0)=q^3f_{5,5,1}(q^8,q^5;q)+f_{5,5,1}(q^3,q^2;q)=G_{1}(0).\qedhere
\end{equation*}
\end{proof}


\subsection{The symmetric double-sum evaluations} \label{section:f551Evaluations}  We evaluate the symmetric double-sums in terms of Appell functions and theta functions.

\begin{proof}[Proof of Proposition \ref{proposition:G14-fullAppellTheta}]

The proofs are all similar, so we will only do the details for the first two expansions.  We first prove (\ref{equation:G0-fullAppellTheta}).  We consider the form
\begin{equation*}
 (q)_{\infty}^3 \mathcal{C}_{2k,0}^{1/2}(q)
 =q^{k}f_{5,5,1}(q,q^{4k-2};q)+q^{3k}f_{5,5,1}(q^6,q^{4k+1};q),
\end{equation*}
We consider the first double-sum.  Using Lemma \ref{lemma:f551-expansion}, we have
\begin{align}
h_{5,5,1}(q,q^{4k-2},-1,-1;q)
&=j(q;q^5)m(-q^{1+4k},-1;q^4)+j(q^{4k-2};q)m(-q^{21-20k},-1;q^{20})
\notag \\
&=j(q;q^5)m(-q^{1+4k},-1;q^4),
\label{equation:G0-fullAppellThetaMsum1}
\end{align}
and
\begin{align}
h_{5,5,1}(q^6,q^{4k+1},-1,-1;q)
&=j(q^6;q^5)m(-q^{4k-1},-1;q^4)+j(q^{4k+1};q)m(-q^{11-20k},-1;q^{20})
\notag \\
&=-q^{-1}j(q;q^5)m(-q^{4k-1},-1;q^4).
\label{equation:G0-fullAppellThetaMsum2}
\end{align}
For the first double-sum, we have
\begin{align*}
f_{5,5,1}&(q,q^{4k-2};q)-h_{5,5,1}(q,q^{4k-2},-1,-1;q)\\
&=-\frac{J_{20}^3}{\overline{J}_{0,4}\overline{J}_{0,20}}
\sum_{d=0}^{4}
q^{2d(d+1)}\frac{j\big (q^{2+4d+4k};q^{5}\big )  j\big (-q^{19-4d-4k};q^{20}\big )j\big (q^{22+4d-16k};q^{20}\big )}
{j\big (q^{21-20k};q^{20})j(q^{1+4d+4k};q^{20}\big )}.
\end{align*}
The sum is over residue classes of $d$ mod $5$, so we can replace $d$ with $d-k$ and use (\ref{equation:j-elliptic}) to get
{\allowdisplaybreaks \begin{align*}
f_{5,5,1}&(q,q^{4k-2};q)-h_{5,5,1}(q,q^{4k-2},-1,-1;q)\\
&=-\frac{J_{20}^3}{\overline{J}_{0,4}\overline{J}_{0,20}}
\sum_{d=0}^{4}
q^{2(d-k)(d-k+1)}\frac{j\big (q^{2+4d};q^{5}\big )  j\big (-q^{19-4d};q^{20}\big )j\big (q^{22+4d-20k};q^{20}\big )}
{j\big (q^{21-20k};q^{20})j(q^{1+4d};q^{20}\big )}\\
&=-\frac{J_{20}^3}{\overline{J}_{0,4}\overline{J}_{0,20}}
\sum_{d=0}^{4}(-1)^{k-1}q^{-20\binom{-k+1}{2}+(k-1)(2+4d)}
(-1)^{k-1}q^{20\binom{-k+1}{2}-(k-1)}
q^{2(d-k)(d-k+1)}\\
&\qquad \qquad \times \frac{j\big (q^{2+4d};q^{5}\big )  j\big (-q^{19-4d};q^{20}\big )j\big (q^{2+4d};q^{20}\big )}
{j\big (q;q^{20})j(q^{1+4d};q^{20}\big )}\\
&=-q^{2k^2-k-1}\frac{J_{20}^3}{\overline{J}_{0,4}\overline{J}_{0,20}J_{1,20}}
\sum_{d=0}^{4}q^{2d^2-2d}
\frac{j\big (q^{2+4d};q^{5}\big )  j\big (-q^{19-4d};q^{20}\big )j\big (q^{2+4d};q^{20}\big )}
{j(q^{1+4d};q^{20}\big )}.
\end{align*}}%
Using (\ref{equation:j-flip}), we arrive at
{\allowdisplaybreaks \begin{align}
f_{5,5,1}&(q,q^{4k-2};q)-h_{5,5,1}(q,q^{4k-2},-1,-1;q)
\label{equation:G0-fullAppellThetaSum1}\\
&=-q^{2k^2-k-1}
\frac{J_{20}^3}{\overline{J}_{0,4}\overline{J}_{0,20}J_{1,20}}
\sum_{d=0}^{4}q^{2d(d-1)}
\frac{j\big (q^{2+4d};q^{5}\big )  j\big (-q^{1+4d};q^{20}\big )j\big (q^{2+4d};q^{20}\big )}
{j(q^{1+4d};q^{20}\big )}.\notag
\end{align}}%

For the second double-sum, we again replace $d$ with $d-k$, and then use (\ref{equation:j-elliptic}) and (\ref{equation:j-flip}) to produce
{\allowdisplaybreaks \begin{align*}
f_{5,5,1}&(q^6,q^{4k+1};q) - h_{5,5,1}(q^6,q^{4k+1},-1,-1;q)\\
&=-\frac{J_{20}^3}{\overline{J}_{0,4}\overline{J}_{0,20}}
\sum_{d=0}^{4}
q^{2d(d+1)}\frac{j\big (q^{5+4d+4k};q^{5}\big )  j\big (-q^{21-4d-4k};q^{20}\big )j\big (q^{10+4d-16k};q^{20}\big )}
{j\big (q^{11-20k};q^{20})j(q^{-1+4d+4k};q^{20}\big )}\\
&=-\frac{J_{20}^3}{\overline{J}_{0,4}\overline{J}_{0,20}}
\sum_{d=0}^{4}
q^{2(d-k)(d-k+1)}\frac{j\big (q^{5+4d};q^{5}\big )  j\big (-q^{21-4d};q^{20}\big )j\big (q^{10+4d-20k};q^{20}\big )}
{j\big (q^{11-20k};q^{20})j(q^{-1+4d};q^{20}\big )}\\
&=-\frac{J_{20}^3}{\overline{J}_{0,4}\overline{J}_{0,20}}
\sum_{d=0}^{4}(-1)^{k}q^{-20\binom{-k}{2}+k(10+4d)}(-1)^kq^{20\binom{-k}{2}-k(11)}
q^{2(d-k)(d-k+1)}\\
&\qquad \qquad \times \frac{j\big (q^{5+4d};q^{5}\big )  j\big (-q^{21-4d};q^{20}\big )j\big (q^{10+4d};q^{20}\big )}
{j\big (q^{11};q^{20})j(q^{-1+4d};q^{20}\big )}\\
&=-q^{2k^2-3k}\frac{J_{20}^3}{\overline{J}_{0,4}\overline{J}_{0,20}J_{11,20}}
\sum_{d=0}^{4}q^{2d(d+1)}
 \frac{j\big (q^{5+4d};q^{5}\big )  j\big (-q^{21-4d};q^{20}\big )j\big (q^{10+4d};q^{20}\big )}
{j(q^{-1+4d};q^{20}\big )}\\
&=q^{2k^2-3k}\frac{J_{20}^3}{\overline{J}_{0,4}\overline{J}_{0,20}}
\sum_{d=0}^{4}q^{2d(d-1)}
 \frac{j\big (q^{4d};q^{5}\big )  j\big (-q^{-1+4d};q^{20}\big )j\big (q^{10+4d};q^{20}\big )}
{J_{11,20}j(q^{-1+4d};q^{20}\big )}.
\end{align*}}%
Again noting that the sums are over residue classes of $d$ mod $5$, we can replace $d$ with $d+1$ to arrive at
{\allowdisplaybreaks \begin{align}
f_{5,5,1}&(q^6,q^{4k+1};q) - h_{5,5,1}(q^6,q^{4k+1},-1,-1;q)
\label{equation:G0-fullAppellThetaSum2}\\
&=q^{2k^2-3k}
\frac{J_{20}^3}{\overline{J}_{0,4}\overline{J}_{0,20}J_{11,20}}
\sum_{d=0}^{4}q^{2d(d+1)}
 \frac{j\big (q^{4+4d};q^{5}\big )  j\big (-q^{3+4d};q^{20}\big )j\big (q^{14+4d};q^{20}\big )}
{j(q^{3+4d};q^{20}\big )}.\notag
\end{align}}%
Combining (\ref{equation:G0-fullAppellThetaMsum1}), (\ref{equation:G0-fullAppellThetaMsum2}),  (\ref{equation:G0-fullAppellThetaSum1}), and (\ref{equation:G0-fullAppellThetaSum2}) gives (\ref{equation:G0-fullAppellTheta})


We prove (\ref{equation:G2-fullAppellTheta}).  We consider the form
\begin{equation*}
 (q)_{\infty}^3 \mathcal{C}_{2k,2}^{1/2}(q)
 =q^{k-1}f_{5,5,1}(q^3,q^{4k};q)+q^{3k+1}f_{5,5,1}(q^{8},q^{4k+3};q).
\end{equation*}
Using Lemma \ref{lemma:f551-expansion}, we have
\begin{align}
h_{5,5,1}(q^3,q^{4k},-1,-1;q)
&=j(q^3;q^5)m(-q^{1+4k},-1;q^4)+j(q^{4k};q)m(-q^{13-20k},-1;q^{20})
\notag \\
&=j(q^3;q^5)m(-q^{1+4k},-1;q^4)
\label{equation:G2-fullAppellThetaMsum1}
\end{align}
and
\begin{align}
h_{5,5,1}(q^{8},q^{4k+3},-1,-1;q)
&=j(q^8;q^5)m(-q^{-1+4k},-1;q^4)+j(q^{4k+3};q)m(-q^{3-20k},-1;q^{20})
\notag \\
&=-q^{-3}j(q^3;q^5)m(-q^{-1+4k},-1;q^4).
\label{equation:G2-fullAppellThetaMsum2}
\end{align}
For the first double-sum, we then have
\begin{align*}
f_{5,5,1}&(q^3,q^{4k},-1,-1;q)-h_{5,5,1}(q^3,q^{4k},-1,-1;q)\\
&=-\frac{J_{20}^3}{\overline{J}_{0,4}\overline{J}_{0,20}}\sum_{d=0}^{4}
q^{2d(d+1)}\frac{j\big (q^{4+4d+4k};q^{5}\big )  j\big (-q^{19-4d-4k};q^{20}\big )j\big (q^{14+4d-16k};q^{20}\big )}
{j\big (q^{13-20k};q^{20})j(q^{1+4d+4k};q^{20}\big )}.
\end{align*}
We note that the sum is over residue classes of $d$ mod $5$, so we can  replacing $d$ with $d-k$ and use (\ref{equation:j-elliptic}) to get
{\allowdisplaybreaks \begin{align*}
f_{5,5,1}&(q^3,q^{4k},-1,-1;q)-h_{5,5,1}(q^3,q^{4k},-1,-1;q)\\
&=-\frac{J_{20}^3}{\overline{J}_{0,4}\overline{J}_{0,20}}\sum_{d=0}^{4}
q^{2(d-k)(d-k+1)}\frac{j\big (q^{4+4d};q^{5}\big )  j\big (-q^{19-4d};q^{20}\big )j\big (q^{14+4d-20k};q^{20}\big )}
{j\big (q^{13-20k};q^{20})j(q^{1+4d};q^{20}\big )}\\
&=-\frac{J_{20}^3}{\overline{J}_{0,4}\overline{J}_{0,20}}\sum_{d=0}^{4}(-1)^{k}q^{-20\binom{-k}{2}}q^{k(14+4d)}
(-1)^{k}q^{20\binom{-k}{2}-13k}
q^{2(d-k)(d-k+1)}\\
&\qquad \qquad \times \frac{j\big (q^{4+4d};q^{5}\big )  j\big (-q^{19-4d};q^{20}\big )j\big (q^{14+4d};q^{20}\big )}
{j\big (q^{13};q^{20})j(q^{1+4d};q^{20}\big )}\\
&=-q^{2k^2-k}\frac{J_{20}^3}{\overline{J}_{0,4}\overline{J}_{0,20}J_{13,20}}\sum_{d=0}^{4}q^{2d(d+1)}
\frac{j\big (q^{4+4d};q^{5}\big )  j\big (-q^{19-4d};q^{20}\big )j\big (q^{14+4d};q^{20}\big )}
{j(q^{1+4d};q^{20}\big )}.
\end{align*}}%
We then use (\ref{equation:j-flip}) to arrive at
{\allowdisplaybreaks \begin{align}
f_{5,5,1}&(q^3,q^{4k},-1,-1;q)-h_{5,5,1}(q^3,q^{4k},-1,-1;q)
\label{equation:G2-fullAppellThetaSum1}\\
&=-q^{2k^2-k}\frac{J_{20}^3}{\overline{J}_{0,4}\overline{J}_{0,20}}\sum_{d=0}^{4}q^{2d(d+1)}
\frac{j\big (q^{4+4d};q^{5}\big )  j\big (-q^{1+4d};q^{20}\big )j\big (q^{14+4d};q^{20}\big )}
{J_{7,20}j(q^{1+4d};q^{20}\big )}.\notag 
\end{align}}%

For the other double-sum, we argue in the same way to obtain
{\allowdisplaybreaks \begin{align*}
f_{5,5,1}&(q^8,q^{4k+3},-1,-1;q)-h_{5,5,1}(q^8,q^{4k+3},-1,-1;q)\\
&=-\frac{J_{20}^3}{\overline{J}_{0,4}\overline{J}_{0,20}}\sum_{d=0}^{4}
q^{2d(d+1)}\frac{j\big (q^{7+4d+4k};q^{5}\big )  j\big (-q^{21-4d-4k};q^{20}\big )j\big (q^{2+4d-16k};q^{20}\big )}
{j\big (q^{3-20k};q^{20})j(q^{-1+4d+4k};q^{20}\big )}\\
&=-\frac{J_{20}^3}{\overline{J}_{0,4}\overline{J}_{0,20}}\sum_{d=0}^{4}
q^{2(d-k)(d-k+1)}\frac{j\big (q^{7+4d};q^{5}\big )  j\big (-q^{21-4d};q^{20}\big )j\big (q^{2+4d-20k};q^{20}\big )}
{j\big (q^{3-20k};q^{20})j(q^{-1+4d};q^{20}\big )}\\
&=-\frac{J_{20}^3}{\overline{J}_{0,4}\overline{J}_{0,20}}
\sum_{d=0}^{4}
(-1)^{k}q^{-20\binom{-k}{2}+k(2+4d)}(-1)^{k}q^{20\binom{-k}{2}-3k}
q^{2(d-k)(d-k+1)}\\
&\qquad \qquad \times \frac{j\big (q^{7+4d};q^{5}\big )  j\big (-q^{21-4d};q^{20}\big )j\big (q^{2+4d};q^{20}\big )}
{j\big (q^{3};q^{20})j(q^{-1+4d};q^{20}\big )}\\
&=-q^{2k^2-3k}\frac{J_{20}^3}{\overline{J}_{0,4}\overline{J}_{0,20}J_{3,20}}
\sum_{d=0}^{4}
q^{2d(d+1)} \frac{j\big (q^{7+4d};q^{5}\big )  j\big (-q^{21-4d};q^{20}\big )j\big (q^{2+4d};q^{20}\big )}
{j(q^{-1+4d};q^{20}\big )}.
\end{align*}}%
Using (\ref{equation:j-elliptic}) and replacing $d$ with $d+1$ yields
{\allowdisplaybreaks \begin{align*}
f_{5,5,1}&(q^8,q^{4k+3},-1,-1;q)-h_{5,5,1}(q^8,q^{4k+3},-1,-1;q)\\
&=q^{2k^2-3k}\frac{J_{20}^3}{\overline{J}_{0,4}\overline{J}_{0,20}J_{3,20}}
\sum_{d=0}^{4}
q^{2d(d-1)-2} \frac{j\big (q^{2+4d};q^{5}\big )  j\big (-q^{-1+4d};q^{20}\big )j\big (q^{2+4d};q^{20}\big )}
{j(q^{-1+4d};q^{20}\big )}\\
&=q^{2k^2-3k}\frac{J_{20}^3}{\overline{J}_{0,4}\overline{J}_{0,20}J_{3,20}}
\sum_{d=0}^{4}
q^{2d(d+1)-2} \frac{j\big (q^{6+4d};q^{5}\big )  j\big (-q^{3+4d};q^{20}\big )j\big (q^{6+4d};q^{20}\big )}
{j(q^{3+4d};q^{20}\big )}.
\end{align*}}%
Using (\ref{equation:j-elliptic}) one last time brings us to 
{\allowdisplaybreaks \begin{align}
f_{5,5,1}&(q^8,q^{4k+3},-1,-1;q)-h_{5,5,1}(q^8,q^{4k+3},-1,-1;q)
\label{equation:G2-fullAppellThetaSum2}\\
&=-q^{2k^2-3k-3}\frac{J_{20}^3}{\overline{J}_{0,4}\overline{J}_{0,20}}
\sum_{d=0}^{4}
q^{2d(d-1)} \frac{j\big (q^{1+4d};q^{5}\big )  j\big (-q^{3+4d};q^{20}\big )j\big (q^{6+4d};q^{20}\big )}
{J_{3,20}j(q^{3+4d};q^{20}\big )}.\notag
\end{align}}%
Combining (\ref{equation:G2-fullAppellThetaMsum1}), (\ref{equation:G2-fullAppellThetaMsum2}),  (\ref{equation:G2-fullAppellThetaSum1}), and (\ref{equation:G2-fullAppellThetaSum2}) gives (\ref{equation:G2-fullAppellTheta}). \qedhere
\end{proof}


\section{The false theta function expansion of negative level string functions}\label{section:generalNegativeLevel}

Proposition \ref{proposition:modStringFnHeckeForm} and Theorem \ref{theorem:negDisc} immediately give 
\begin{align*}
(q)_{\infty}^{3}\mathcal{C}_{m,\ell}^{N}(q)
&=\frac{1}{2}\sum_{t=0}^{2pp^{\prime}-1}(-1)^tq^{(\frac{m-\ell}{2})t}q^{\binom{t}{2}}\\
&\qquad \times \left ( q^{(1+\ell) t}
j(-q^{p^{\prime}t+p(p^{\prime}+\ell+1)};q^{2pp^{\prime}})
-j(-q^{p^{\prime}t+p(p^{\prime}-(\ell+1))};q^{2pp^{\prime}})\right ) \\
&\qquad \qquad \times \sum_{r\in\mathbb{Z}}\sg(r)
\left (q^{(pp^{\prime})^2N+pp^{\prime}m-tpp^{\prime}N}\right )^r
q^{-2(pp^{\prime})^2N\binom{r+1}{2}}.
\end{align*}
We write $t=2p i + k$, $0\le i \le p^{\prime}-1$, $0\le k\le 2p-1$.  Hence
\begin{equation*}
\sum_{t=0}^{2pp^{\prime}-1}
=\sum_{i=0}^{p^{\prime}-1}\sum_{k=0}^{2p-1}.
\end{equation*}
Upon substitution, we then have
{\allowdisplaybreaks \begin{align*}
(q)_{\infty}^{3}\mathcal{C}_{m,\ell}^{N}(q)
&=\frac{1}{2}\sum_{i=0}^{p^{\prime}-1}\sum_{k=0}^{2p-1}(-1)^{2pi+k}q^{(\frac{m-\ell}{2})(2pi+k)}q^{\binom{2pi+k}{2}}\\
&\qquad \times \Big ( q^{(1+\ell) (2pi+k)}j(-q^{p^{\prime}(2pi+k)+p(p^{\prime}+\ell+1)};q^{2pp^{\prime}})\\
&\qquad \qquad -j(-q^{p^{\prime}(2pi+k)+p(p^{\prime} -(\ell+1))};q^{2pp^{\prime}})\Big ) \\
&\qquad \qquad \times \sum_{r\in\mathbb{Z}}\sg(r)
\left (q^{(pp^{\prime})^2N+pp^{\prime}m-(2pi+k)pp^{\prime}N}\right )^r
q^{-2(pp^{\prime})^2N\binom{r+1}{2}}.
\end{align*}}%
Using the quasi-elliptic transformation property (\ref{equation:j-elliptic}) allows us to pull the $i$ outside of the theta function.  This gives
{\allowdisplaybreaks \begin{align*}
(q)_{\infty}^{3}\mathcal{C}_{m,\ell}^{N}(q)
&=\frac{1}{2}\sum_{k=0}^{2p-1}\sum_{i=0}^{p^{\prime}-1}(-1)^{k}
q^{(\frac{m-\ell}{2})2pi}
q^{(\frac{m-\ell}{2})k}
q^{\binom{2pi}{2}+2pik+\binom{k}{2}}\\
&\qquad \times \Big ( q^{2pi+k+\ell 2pi+\ell k}
q^{-2pp^{\prime}\binom{i}{2}-i(kp^{\prime}+p(p^{\prime}+\ell+1))}
j(-q^{p^{\prime}k+p(p^{\prime}+\ell+1)};q^{2pp^{\prime}})\\
&\qquad \qquad 
-q^{-2pp^{\prime}\binom{i}{2}-i(kp^{\prime}+p(p^{\prime}-(\ell+1)))}
j(-q^{p^{\prime}k+p(p^{\prime} -(\ell+1))};q^{2pp^{\prime}})\Big ) \\
&\qquad \qquad \times \sum_{r\in\mathbb{Z}}\sg(r)
\left (q^{(pp^{\prime})^2N+pp^{\prime}m-(2pi+k)pp^{\prime}N}\right )^r
q^{-2(pp^{\prime})^2N\binom{r+1}{2}}.
\end{align*}}%
Now we pull out the terms not dependent on $r$ or $i$.  This reads
{\allowdisplaybreaks \begin{align*}
(q)_{\infty}^{3}\mathcal{C}_{m,\ell}^{N}(q)
&=\frac{1}{2}\sum_{k=0}^{2p-1}
(-1)^{k} 
q^{(\frac{m-\ell}{2})k}q^{\binom{k}{2}}q^{(1+\ell )k}
j(-q^{p^{\prime}k+p(p^{\prime}+\ell+1)};q^{2pp^{\prime}})\\
&\qquad \qquad \times\sum_{i=0}^{p^{\prime}-1}
q^{(1+\ell )2pi}q^{(\frac{m-\ell}{2})2pi}q^{\binom{2pi}{2}+2pik} 
q^{-2pp^{\prime}\binom{i}{2}-i(kp^{\prime}+p(p^{\prime}+\ell+1))} \\
&\qquad \qquad \qquad \times \sum_{r\in\mathbb{Z}}\sg(r)
\left (q^{(pp^{\prime})^2N+pp^{\prime}m-(2pi+k)pp^{\prime}N}\right )^r
q^{-2(pp^{\prime})^2N\binom{r+1}{2}}\\
&\qquad -\frac{1}{2}\sum_{k=0}^{2p-1}
(-1)^{k}q^{(\frac{m-\ell}{2})k}q^{\binom{k}{2}}j(-q^{p^{\prime}k+p(p^{\prime} -(\ell+1))};q^{2pp^{\prime}})\\
&\qquad \qquad \times 
\sum_{i=0}^{p^{\prime}-1}
q^{(\frac{m-\ell}{2})2pi}q^{\binom{2pi}{2}+2pik} 
q^{-2pp^{\prime}\binom{i}{2}-i(kp^{\prime}+p(p^{\prime}-(\ell+1)))} \\
&\qquad \qquad \qquad \times \sum_{r\in\mathbb{Z}}\sg(r)
\left (q^{(pp^{\prime})^2N+pp^{\prime}m-(2pi+k)pp^{\prime}N}\right )^r
q^{-2(pp^{\prime})^2N\binom{r+1}{2}}.
\end{align*}}%
We then regroup terms in the exponents according to $R:=p^{\prime}r+i$ and recall that $N=\frac{p^{\prime}}{p}-2$, so $pN=p^{\prime}-2p$.  This brings us to a nice mid-way point:
{\allowdisplaybreaks \begin{align}
(q)_{\infty}^{3}\mathcal{C}_{m,\ell}^{N}(q)
&=-\frac{1}{2}\sum_{k=0}^{2p-1}
(-1)^{k}q^{(\frac{m-\ell}{2})k}q^{\binom{k}{2}}
\label{equation:genFalseThmMidWayPoint}\\
&\qquad \times \left ( j(-q^{p^{\prime}k+p(p^{\prime} -(\ell+1))};q^{2pp^{\prime}})
-q^{(1+\ell)k}
j(-q^{p^{\prime}k+p(p^{\prime}+\ell+1)};q^{2pp^{\prime}})\right)
\notag \\ 
&\qquad \times
\sum_{R\in\mathbb{Z}}\sg(R)
q^{-p^2NR^2+pR(m-kN)}.\notag
\end{align}}%


We work on rewriting (\ref{equation:genFalseThmMidWayPoint}).  We first show that for $k=0$ and $k=p$ that the theta expression vanishes.  For $k=0$, we use (\ref{equation:j-flip}) to see
\begin{align*}
& j(-q^{p^{\prime}k+p(p^{\prime} -(\ell+1))};q^{2pp^{\prime}})
-q^{(1+\ell)k}
j(-q^{p^{\prime}k+p(p^{\prime}+\ell+1)};q^{2pp^{\prime}})\\
&\qquad \to j(-q^{p(p^{\prime} -(\ell+1))};q^{2pp^{\prime}})
-j(-q^{p(p^{\prime}+\ell+1)};q^{2pp^{\prime}})\\
&\qquad \qquad = j(-q^{p(p^{\prime} +(\ell+1))};q^{2pp^{\prime}})
-j(-q^{p(p^{\prime}+\ell+1)};q^{2pp^{\prime}})=0.
\end{align*}
For $k=p$, we use (\ref{equation:j-flip}) and (\ref{equation:j-elliptic}) to get
\begin{align*}
& j(-q^{p^{\prime}k+p(p^{\prime} -(\ell+1))};q^{2pp^{\prime}})
-q^{(1+\ell)k}
j(-q^{p^{\prime}k+p(p^{\prime}+\ell+1)};q^{2pp^{\prime}})\\
&\to j(-q^{2pp^{\prime} -p(\ell+1)};q^{2pp^{\prime}})
-q^{(1+\ell)p}
j(-q^{2pp^{\prime}+p(\ell+1)};q^{2pp^{\prime}})\\
&\qquad = j(-q^{p(\ell+1)};q^{2pp^{\prime}})
-q^{(1+\ell)p}q^{-p(1+\ell)}
j(-q^{p(\ell+1)};q^{2pp^{\prime}})=0.
\end{align*}
Hence we can rewrite (\ref{equation:genFalseThmMidWayPoint}) as
{\allowdisplaybreaks \begin{align}
(q)_{\infty}^{3}\mathcal{C}_{m,\ell}^{N}(q)
&=-\frac{1}{2}\left ( \sum_{k=1}^{p-1}+\sum_{k=p+1}^{2p-1}\right ) 
(-1)^{k}q^{(\frac{m-\ell}{2})k}q^{\binom{k}{2}}
\label{equation:genFalseThmMidWayPoint2}\\
&\qquad \times \left ( j(-q^{p^{\prime}k+p(p^{\prime} -(\ell+1))};q^{2pp^{\prime}})
-q^{(1+\ell)k}
j(-q^{p^{\prime}k+p(p^{\prime}+\ell+1)};q^{2pp^{\prime}})\right)
\notag\\ 
&\qquad \times
\sum_{R\in\mathbb{Z}}\sg(R)
q^{-p^2NR^2+pR(m-kN)}.\notag
\end{align}}%
Let us focus on the second sum, which we will denote by
\begin{align*}
A(m,\ell,p,p^{\prime})
&:=\sum_{k=p+1}^{2p-1}(-1)^{k}
q^{(\frac{m-\ell}{2})k}q^{\binom{k}{2}}\\
&\qquad \times \left ( j(-q^{p^{\prime}k+p(p^{\prime} -(\ell+1))};q^{2pp^{\prime}})
-q^{(1+\ell)k}
j(-q^{p^{\prime}k+p(p^{\prime}+\ell+1)};q^{2pp^{\prime}})\right) \\ 
&\qquad \times
\sum_{R\in\mathbb{Z}}\sg(R)
q^{-p^2NR^2+pR(m-kN)}.
\end{align*}
We replace $k$ with $2p-k$ and then apply (\ref{equation:j-elliptic}).  This yields
{\allowdisplaybreaks \begin{align*}
A(m,\ell,p,p^{\prime}) &=\sum_{k=1}^{p-1} 
(-1)^{2p-k}q^{(\frac{m-\ell}{2})(2p-k)}q^{\binom{2p-k}{2}}\\
&\qquad \times \left ( j(-q^{2pp^{\prime}-p^{\prime}k+p(p^{\prime} -(\ell+1))};q^{2pp^{\prime}})
-q^{(1+\ell)(2p-k)}
j(-q^{2pp^{\prime}-p^{\prime}k+p(p^{\prime}+\ell+1)};q^{2pp^{\prime}})\right) \\ 
&\qquad \times
\sum_{R\in\mathbb{Z}}\sg(R)
q^{-p^2NR^2+pR(m-(2p-k)N)}\\
&=\sum_{k=1}^{p-1} 
(-1)^{2p-k}q^{(\frac{m-\ell}{2})(2p-k)}q^{\binom{2p-k}{2}}\\
&\qquad \times \Big ( 
q^{p^{\prime}k-p(p^{\prime} -(\ell+1))}
j(-q^{2pp^{\prime}-p^{\prime}k+p(p^{\prime} -(\ell+1))}; q^{2pp^{\prime}})\\
&\qquad \qquad -q^{(1+\ell)(2p-k)}
q^{p^{\prime}k-p(p^{\prime} +(\ell+1))}
j(-q^{2pp^{\prime} -p^{\prime}k+p(p^{\prime} +\ell+1)}; q^{2pp^{\prime}})\Big) \\ 
&\qquad \times
\sum_{R\in\mathbb{Z}}\sg(R)
q^{-p^2NR^2+pR(m-(2p-k)N)}.
\end{align*}}%
We then pull out a common factor and use (\ref{equation:j-flip}) to get 
{\allowdisplaybreaks \begin{align*}
A(m,\ell,p,p^{\prime}) 
&=\sum_{k=1}^{p-1} 
(-1)^{k}q^{(\frac{m-\ell}{2})(2p-k)}q^{\binom{2p-k}{2}}
q^{(1+\ell)(2p-k)}q^{p^{\prime}k-p(p^{\prime}+\ell+1)}\\
&\qquad \times \Big ( 
q^{(1+\ell)k}
j(-q^{p^{\prime}k+p(p^{\prime} +(\ell+1))};q^{2pp^{\prime}})
-j(-q^{p^{\prime}k+p(p^{\prime}-(\ell+1))};q^{2pp^{\prime}})\Big) \\ 
&\qquad \times
\sum_{R\in\mathbb{Z}}\sg(R)
q^{-p^2NR^2+pR(m-(2p-k)N)}.
\end{align*}}%
Rewriting the exponents and reordering the terms brings us to the final form:
\begin{align}
A(m,\ell,p,p^{\prime})  
&=-\sum_{k=1}^{p-1} 
(-1)^{k}q^{(\frac{m-\ell}{2})k}q^{\binom{k}{2}}
q^{(k-p)(p^{\prime}-m-2p)}
\label{equation:AfunFinalForm}\\
&\qquad \times \Big ( 
j(-q^{p^{\prime}k+p(p^{\prime}-\ell+1)};q^{2pp^{\prime}})-q^{(1+\ell)k}
j(-q^{p^{\prime}k+p(p^{\prime} +(\ell+1))};q^{2pp^{\prime}})
\Big)
\notag\\ 
&\qquad \times
\sum_{R\in\mathbb{Z}}\sg(R)
q^{-p^2NR^2+pR(m-(2p-k)N)}.
\notag
\end{align}
Substituting (\ref{equation:AfunFinalForm}) into (\ref{equation:genFalseThmMidWayPoint2}) and recalling that $pN=p^{\prime}-2p$ gives the result:
\begin{align*}
(q)_{\infty}^{3}&\mathcal{C}_{m,\ell}^{N}(q)\\
&=- \frac{1}{2}\sum_{k=1}^{p-1} 
(-1)^{k}q^{(\frac{m-\ell}{2})k}q^{\binom{k}{2}}
\left ( j(-q^{p^{\prime}k+p(p^{\prime} -(\ell+1))};q^{2pp^{\prime}})
-q^{(1+\ell)k}
j(-q^{p^{\prime}k+p(p^{\prime}+\ell+1)};q^{2pp^{\prime}})\right) \\ 
&\qquad \times
\left ( \sum_{R\in\mathbb{Z}}\sg(R)
q^{-p^2NR^2+pR(m-kN)}
-q^{(k-p)(pN-m)}
\sum_{R\in\mathbb{Z}}\sg(R)
q^{-p^2NR^2+pR(m-(2p-k)N)}
\right ).
\end{align*}

\section{The false theta function expansion of $(-1/2)$-level string functions}\label{section:fractionalLevelMinus12}

We set $(p,p^{\prime})=(2,3)$, so $N=-1/2$.  Here $\ell\in\{0,1\}$ and $m\equiv \ell \pmod 2$.  Theorem \ref{theorem:genNegativeLevelExpansion} immediately gives

\begin{align*}
(q)_{\infty}^{3}\mathcal{C}_{m,\ell}^{-1/2}(q)
&=\frac{1}{2}q^{(\frac{m-\ell}{2})}
\left ( j(-q^{7 -2\ell};q^{12})
-q^{(1+\ell)}
j(-q^{11+2\ell};q^{12})\right) \\ 
&\qquad \times
\left ( \sum_{R\in\mathbb{Z}}\sg(R)
q^{2R^2+R(2m+1)}
-q^{1+m}
\sum_{R\in\mathbb{Z}}\sg(R)
q^{2R^2+R(2m+3)}
\right ).
\end{align*}
Using (\ref{equation:j-flip}) and then (\ref{equation:jsplit-m2}) yields
\begin{align*}
(q)_{\infty}^{3}\mathcal{C}_{m,\ell}^{-1/2}(q)
&=\frac{1}{2}q^{(\frac{m-\ell}{2})}
\left ( j(-q^{5 +2\ell};q^{12})
-q^{(1+\ell)}
j(-q^{11+2\ell};q^{12})\right) \\ 
&\qquad \times
\left ( \sum_{R\in\mathbb{Z}}\sg(R)
q^{2R^2+R(2m+1)}
-q^{1+m}
\sum_{R\in\mathbb{Z}}\sg(R)
q^{2R^2+R(2m+3)}
\right )\\
&=\frac{1}{2}q^{(\frac{m-\ell}{2})}
j(q^{1+\ell};q^3) \\ 
&\qquad \times
\left ( \sum_{R\in\mathbb{Z}}\sg(R)
q^{\frac{1}{2}2R(2R+2m+1)}
-\sum_{R\in\mathbb{Z}}\sg(R)
q^{\frac{1}{2}(2R+1)(2R+2m+2)}
\right ).
\end{align*}
Which is just
\begin{align*}
(q)_{\infty}^{3}\mathcal{C}_{m,\ell}^{-1/2}(q)
&=\frac{1}{2}q^{(\frac{m-\ell}{2})}
j(q^{1+\ell};q^3)
\sum_{i\in\mathbb{Z}}\sg(i)(-1)^{i}
q^{\frac{1}{2}i(i+2m+1)}.
\end{align*}

We recall that $\ell \in \{ 0,1 \}$.  Thus
 \begin{equation*}
 j(q^{1+\ell};q^3)=j(q;q^3)=j(q^2;q^3)=(q;q^3)_{\infty}(q^2;q^3)_{\infty}(q^3;q^3)_{\infty}=(q;q)_{\infty},
 \end{equation*}
and we arrive at
  \begin{equation*}
 \mathcal{C}_{m,\ell}^{-1/2}(q)
 =\frac{1}{2}\frac{q^{\frac{1}{2}(m-\ell)}}{(q)_{\infty}^2} 
   \sum_{i\in\mathbb{Z}}\sg(i)(-1)^iq^{\frac{1}{2}i(i+2m+1)}.
 \end{equation*}
 
We still have some more work to do.  We write
  {\allowdisplaybreaks \begin{align*}
 \mathcal{C}_{m,\ell}^{-1/2}(q)
 &=\frac{1}{2}\frac{q^{\frac{1}{2}(m-\ell)}}{(q)_{\infty}^2} 
 \left (   \sum_{i\ge0 } (-1)^iq^{\frac{1}{2}i(i+2m+1)} -  \sum_{i<0} (-1)^iq^{\frac{1}{2}i(i+2m+1)}  \right )\\
&=\frac{1}{2}\frac{q^{\frac{1}{2}(m-\ell)}}{(q)_{\infty}^2} 
 \left (   2\sum_{i\ge0 } (-1)^iq^{\frac{1}{2}i(i+2m+1)} 
 - \sum_{i\in\mathbb{Z}} (-1)^iq^{\frac{1}{2}i(i+2m+1)}  \right ).
 \end{align*}}%
Using the Jacobi triple product identity (\ref{equation:JTPid}) and noting that $j(q^n;q)=0$ for all $n\in\mathbb{Z}$ gives the result:
{\allowdisplaybreaks \begin{align*}
 \mathcal{C}_{m,\ell}^{-1/2}(q)
&=\frac{1}{2}\frac{q^{\frac{1}{2}(m-\ell)}}{(q)_{\infty}^2} 
 \left (   2\sum_{i\ge0 } (-1)^iq^{\frac{1}{2}i(i+2m+1)} 
 - \sum_{i\in\mathbb{Z}} (-1)^iq^{\frac{1}{2}i(i-1)+i(m+1)}  \right )\\
&= \frac{1}{2}\frac{q^{\frac{1}{2}(m-\ell)}}{(q)_{\infty}^2} 
 \left (   2\sum_{i\ge0 } (-1)^iq^{\frac{1}{2}i(i+2m+1)} 
 - j(q^{m+1};q) \right )\\
&= \frac{q^{\frac{1}{2}(m-\ell)}}{(q)_{\infty}^2} 
  \sum_{i\ge0 } (-1)^iq^{\frac{1}{2}i(i+2m+1)}. 
\end{align*}}%


\section{The false theta function expansion of $(-2/3)$-level string functions}\label{section:fractionalLevelMinus34}
We set $(p,p^{\prime})=(3,4)$, $N=-2/3$.  Theorem \ref{theorem:genNegativeLevelExpansion} immediately gives
\begin{align*}
(q)_{\infty}^{3}\mathcal{C}_{m,\ell}^{-2/3}(q)
&=- \frac{1}{2}\sum_{k=1}^{2} 
(-1)^{k}q^{(\frac{m-\ell}{2})k}q^{\binom{k}{2}}
\left ( j(-q^{4k+9 -3\ell};q^{24})
-q^{(1+\ell)k}
j(-q^{4k+15+3\ell};q^{24})\right) \\ 
&\qquad \times
\left ( \sum_{R\in\mathbb{Z}}\sg(R)
q^{6R^2+R(3m+2k)}
-q^{(k-3)(-2-m)}
\sum_{R\in\mathbb{Z}}\sg(R)
q^{6R^2+R(3m+12-2k)}
\right ).
\end{align*}
Expanding the sum over $k$ gives
\begin{align*}
(q)_{\infty}^{3}\mathcal{C}_{m,\ell}^{-2/3}(q)
&=\frac{1}{2}q^{(\frac{m-\ell}{2})}
\left ( j(-q^{13 -3\ell};q^{24})
-q^{(1+\ell)}
j(-q^{19+3\ell};q^{24})\right) \\ 
&\qquad \times
\left ( \sum_{R\in\mathbb{Z}}\sg(R)
q^{12R^2+R(3m+2)}
-q^{2(2+m)}
\sum_{R\in\mathbb{Z}}\sg(R)
q^{6R^2+R(3m+10)}
\right )\\
&\qquad -  \frac{1}{2}q^{(m-\ell)+1}
\left ( j(-q^{17 -3\ell};q^{24})
-q^{(1+\ell)2}
j(-q^{23+3\ell};q^{24})\right) \\ 
&\qquad \times
\left ( \sum_{R\in\mathbb{Z}}\sg(R)
q^{6R^2+R(3m+4)}
-q^{2+m}
\sum_{R\in\mathbb{Z}}\sg(R)
q^{12R^2+R(3m+8)}
\right ).
\end{align*}
We rewrite the two theta function differences in order to use the quintuple product identity.  For the first difference, we use (\ref{equation:j-flip}) and then (\ref{equation:quintuple}) to get
\begin{align*}
j(-q^{13 -3\ell};q^{24})-q^{(1+\ell)}j(-q^{19+3\ell};q^{24})
&=j(-q^{11 +3\ell};q^{24})-q^{(1+\ell)}j(-q^{19+3\ell};q^{24})\\
&=\frac{j(q^{1+\ell};q^{8})j(q^{10+2\ell};q^{16})}{J_{16}}.
\end{align*}
Similarly, using (\ref{equation:j-elliptic}) and then (\ref{equation:j-flip}) yields
\begin{align*}
j(-q^{17 -3\ell};q^{24})
-q^{(1+\ell)2}j(-q^{23+3\ell};q^{24})
&=j(-q^{17 -3\ell};q^{24})
-q^{3-\ell}j(-q^{-1+3\ell};q^{24})\\
&=j(-q^{17 -3\ell};q^{24})
-q^{3-\ell}j(-q^{25-3\ell};q^{24}).
\end{align*}
The quintuple product identity and then (\ref{equation:j-flip}) gives
\begin{equation*}
j(-q^{17 -3\ell};q^{24})
-q^{(1+\ell)2}j(-q^{23+3\ell};q^{24})
=\frac{j(q^{3-\ell};q^{8}j(q^{14-2\ell};q^{16}))}{J_{16}}
=\frac{j(q^{5+\ell};q^{8})j(q^{2+2\ell};q^{16})}{J_{16}}.
\end{equation*}
Substituting back in, we obtain our result:
\begin{align*}
(q)_{\infty}^{3}\mathcal{C}_{m,\ell}^{-2/3}(q)
&=q^{(\frac{m-\ell}{2})}
\frac{j(q^{1+\ell};q^{8})j(q^{10+2\ell};q^{16})}{J_{16}} \\ 
&\qquad \times
\left ( \sum_{R\in\mathbb{Z}}\sg(R)
q^{6R^2+R(3m+2)}
-q^{2(2+m)}
\sum_{R\in\mathbb{Z}}\sg(R)
q^{6R^2+R(3m+10)}
\right )\\
&\qquad -  q^{(m-\ell)+1}
\frac{j(q^{5+\ell};q^{8})j(q^{2+2\ell};q^{16})}{J_{16}} \\ 
&\qquad \times
\left ( \sum_{R\in\mathbb{Z}}\sg(R)
q^{6R^2+R(3m+4)}
-q^{2+m}
\sum_{R\in\mathbb{Z}}\sg(R)
q^{6R^2+R(3m+8)}
\right ).
\end{align*}

\section*{Acknowledgements}
This work was supported  by the Ministry of Science and Higher Education of the Russian Federation (agreement no. 075-15-2022-287). The first author was also supported in part by the M{\"o}bius Contest Foundation for Young Scientists.

\end{document}